\documentclass[12pt,oneside,english]{amsart}
\usepackage{mathptmx}

\usepackage[T1]{fontenc}
\usepackage[latin9]{inputenc}
\usepackage[a4paper]{geometry}
\usepackage{amsthm}
\usepackage{amstext}
\usepackage{amssymb}
\usepackage{esint}

\makeatletter
\numberwithin{equation}{section}
\numberwithin{figure}{section}
\theoremstyle{plain}
\newtheorem{thm}{\protect\theoremname}[section]
  \theoremstyle{plain}
  
  \theoremstyle{plain}
  \newtheorem{prop}[thm]{\protect\propositionname}
  \theoremstyle{remark}
  \newtheorem{rem}[thm]{\protect\remarkname}
   \theoremstyle{lem}
  \newtheorem{lem}[thm]{\protect\lemmaname}
  \theoremstyle{definition}
  \newtheorem{definition}[thm]{\protect\defname}
\makeatother

\usepackage{babel}
  \providecommand{\corollaryname}{Corollary}
  \providecommand{\propositionname}{Proposition}
  \providecommand{\remarkname}{Remark}
\providecommand{\theoremname}{Theorem}
\providecommand{\lemmaname}{Lemma}
\providecommand{\defname}{Definition}
\begin{document}

\title{Mabuchi metrics and properness of the modified Ding functional}


\author{Yan Li\ \ \ \ \ \ Bin  $\text{Zhou}^*$}

\address{School of Mathematical Sciences, Peking
University, Beijing 100871, China.}

\email{liyanmath@pku.edu.cn\ \ \ bzhou@pku.edu.cn}

\thanks {*Partially supported by NSFC 11571018 and 11331001}

\begin{abstract}
In this paper, we study Mabuchi metrics on Fano manifolds.
We prove that Mabuchi metrics exist if the modified Ding functional is proper modulo a reductive subgroup of its automorphism group. On the other hand, the inverse that Mabuchi metrics implies the properness is obtained by using Darvas-Rubinstein's properness principle.  As an application,
we establish  a criterion for the existence of Mabuchi metrics on  Fano group compactifications.
\end{abstract}

\maketitle

\section{Introduction}

The existence of canonical metrics has been a fundamental
and longstanding problem in K\"ahler geometry. On Fano manifolds, K\"ahler-Einstein metrics have been studied extensively. The most remarkable progress is the resolution of Yau-Tian-Donaldson conjecture which relates the existence of K\"ahler-Einstein metrics to the K-stability of the Fano manifold
\cite{T4, CDS}. It has been known early in 1980's that the existence of K\"ahler-Einsten metrics fails when the Fano manifold has nonvanishing Futaki invariant.
In this case, other canonical metrics, such as extremal metrics and K\"ahler-Ricci solitons have attracted many attentions.

In \cite{Mab1, Mab3, Mab4, Mab2},
Mabuchi studied a generalized K\"ahler-Einstein metric, which is neither an extremal metric nor a K\"ahler-Ricci soliton.  Following \cite{Yao}, we call this metric the {\it Mabuchi metric} for simplicity. Let $M$ be a compact Fano manifold of complex dimension $n$. Let
$$\omega=\sqrt{-1}g_{i\bar j}dz^i\wedge d\bar z^j\in2\pi c_1(M)$$ be a K\"ahler metric and $h_\omega$ be its Ricci potential.  $\omega$ is  a  Mabuchi metric if
\begin{equation}\label{mab-def}
X_\omega:=-\sqrt{-1}g^{i\bar j}{\frac{\partial e^{h_\omega}}{\partial \bar z^j}}{\frac{\partial}{\partial  z^i}}
\end{equation}
is holomorphic \cite{Mab1}.  The uniqueness of Mabuchi metrics has been proved in \cite{Mab2}.
Recently, Donaldson introduced a new GIT (geometric invariant theory) picture \cite{D}, in which  the corresponding moment map is given by the Ricci potential.
Then Yao observed that in this picture
$X_\omega$ is holomorphic if and only if $\omega$ is a critical point of the norm square of the moment map,  given by the following energy  \cite{Yao}
\begin{equation}
\mathcal E^D(\omega)=\int_M (e^{h_\omega}-1)^2\omega^n.
\end{equation}
This brings new interests in the study of Mabuchi metrics.
On toric Fano manifolds, the notion of relative
Ding stability has been introduced by Yao \cite{Yao}. He has also established the existence of Mabuchi metrics when the toric Fano manifold is relatively Ding stable.
The purpose of this paper is to discuss the existence of Mabuchi metrics on general Fano manifolds
through properness of energy functionals.

According to \cite{Mab1}, if $\omega$ is a Mabuchi metric, then \eqref{mab-def} coincides with the
extremal vector field \cite{Mabuchi-Futaki}. To state the main results, we first recall notions on extremal
vector field.
 Denote by $Aut_0(M)$ the identity component of its holomorphic transformation group.
Its Lie algebra $\eta(M)$ consists of all holomorphic vector fields on $M$.
 $Aut_0(M)$ admits a semi-direct decomposition $$Aut_0(M)=Aut_r(M)\propto R_u,$$
where $Aut_r(M)\subset Aut_0(M)$ is a reductive group and $R_u$ is the unipotent radical of $Aut_0(M)$. 
Denote by $\eta_r(M)$ the reductive part of $\eta(M)$.
For any $v\in\eta(M)$, let $K_v$ be the one parameter group generated by the image part Im$(v)$. For a K\"ahler metric $\omega_0\in2\pi c_1(M)$, by Hodge theorem,
there is a unique normalized potential given by
\begin{eqnarray}\label{0201+}
i_v\omega_0=\sqrt{-1}\bar\partial \theta_v(\omega_0),~\int_{M}\theta_v(\omega_0)\omega_0^n=0.
\end{eqnarray}
Then $\theta_v(\omega)$ is real valued if and only if $\omega$ is $K_v$-invariant. For any
$$\phi\in\mathcal H_{v}(\omega_0):=\{\phi\in C^{\infty}(M)|\omega_{\phi}:=\omega_0+\sqrt{-1}\partial\bar\partial\phi>0,\ \phi\text{ is $K_v$-invariant}\},$$
the normalized potential
$\theta_v(\omega_{\phi})=\theta_v(\omega_0)+v(\phi)$.
Denote by $Fut(v)$ the Futaki invariant of $v\in \eta(M)$.
The extremal vector field $X$ is the holomorphic vector field uniquely determined by \cite{Mabuchi-Futaki}
\begin{eqnarray}\label{0201}
Fut_X(v):=Fut(v)+\int_M\theta_v(\omega_0)\theta_X(\omega_0)\omega_0^n=0,~\forall v\in\eta(M).
\end{eqnarray}
Moreover, $X\in\eta_c(M)$, the centre of $\eta_r(M)$ and $K_X$ lies in a compact Lie group.

From now on, we assume that $\omega_0$ is $K_X$-invariant unless otherwise claimed. As pointed by Mabuchi \cite{Mab2}, both $\displaystyle\min_M\theta_X(\omega_\phi)$ and $\displaystyle\max_M\theta_X(\omega_\phi)$ are independent of the choice of $\omega_\phi\in2\pi c_1(M)$. For convenience, we denote by
$$c_X:=\min_M\{1-\theta_X(\omega_\phi)\}, ~C_X:=\max_M\{1-\theta_X(\omega_\phi)\}.$$
By \cite{Mab1}, Mabuchi metrics exist only if $c_X>0$, and $\omega_\phi\in 2\pi c_1(M)$ is a  Mabuchi metric if
\begin{equation}\label{mabuchi-def}
Ric(\omega_\phi)-\omega_\phi=\sqrt{-1}\partial\bar\partial\log(1-\theta_X(\omega_\phi)).
\end{equation}

In \cite{Tian97},  Tian introduced the notion of
properness of energy functionals as an analytic characterization of existence of K\"ahler-Einsten metrics.
When the automorphism group of $M$ is not discret, a notion of  properness modulo a subgroup of $Aut_0(M)$ was reformulated \cite{CTZ, DR, Tian, ZZ}. In particular, Darvas-Rubinstein established a properness principle and solved Tian's properness conjecture \cite{DR}. It is natural to ask the analogous problem for
Mabuchi metrics.
By \cite{Mab3}, the Mabuchi metric is a critical point of the following {\it modified Ding functional}
\begin{eqnarray}\label{0203}
\mathcal D_{X}(\phi)=-{\frac1V}\int_0^1\int_M\dot\phi_s(1-\theta_X(\omega_{\phi_s}))\omega_{\phi_s}^n\wedge ds-\log\left({\frac1V}\int_M e^{h_{0}-\phi}\omega_0^n\right),
\end{eqnarray}
where $V=\int_M\omega_0^n$, $\{\phi_s\}_{s\in[0,1]}$ is any smooth path in $\mathcal H_{X}(\omega_0)$ joining $0$ and $\phi$, and  $h_0$ is the Ricci potential of $\omega_0$, normalized by $$\int_Me^{h_0}\omega_0^n=\int_M\omega_0^n.$$
Our first main result is the following properness theorem.
\begin{thm}\label{0301}
If $c_X>0$ and the modified Ding functional is proper\footnote{For the definition of properness, see Definition \ref{prop def} below.}  modulo a reductive subgroup $H^c$of $Aut_0(M)$ which contains $K_X$, then $M$ admits Mabuchi metrics.
\end{thm}

It is also interesting to ask the inverse of this theorem.
One can show that the existence of Mabuchi metric implies the properness of  $\mathcal D_X(\cdot)$ modulo the automorphism group of $M$ following the arguments for K\"ahler-Ricci solitons  \cite{CTZ}. However, an optimal properness can be obtained by using the properness principle of Darvas-Rubinstein \cite{DR}.

\begin{thm}\label{inverse proper}
Suppose $Aut_0^X(M)$ is the centralizer of $K_X^c$ in $Aut_0(M)$.
If $M$ admits Mabuchi metrics, then there exists $C, C'>0$, such that
\begin{eqnarray*}
\mathcal D_{X}(\phi)\geq C\inf_{\sigma\in Aut_0^X(M)} J_{X}(\phi_\sigma)-C', \
\forall \phi\in \mathcal H_{X}(\omega_0),
\end{eqnarray*}
where $J_X$ is the modified Aubin's functional (see Section 2.1) $\phi_\sigma$ is defined by
$$\sigma^*(\omega_\phi)=\omega_0+\sqrt{-1}\partial\bar\partial\phi_\sigma.$$
\end{thm}

\begin{rem}
Suppose $\omega_0$ is a Mabuchi metric on $M$. We can define
\begin{equation}\label{operator}
\Lambda_{1, X}=\{u\in C^\infty(M) | \triangle_{\omega_0}u-\frac{X}{1-\theta_X(\omega_0)}u=-u\}.
\end{equation}
Then by the similar argument as \cite[Lemma 3.2]{WZZ},
one can show that the properness modulo $Aut_0^X(M)$ is equivalent to
the properness for K\"ahler potentials that are perpendicular to
$\Lambda_{1, X}$  with respect to the weighted inner product
$$(\varphi, \psi)=\int_M \varphi\psi (1-\theta_X(\omega_0)) \omega_0^n.$$
\end{rem}

\vskip 5pt

The properness condition can be verified for some special Fano manifolds. A characterization for the properness of the modified Ding functional on toric Fano manifolds has been given by \cite{Na}.  We consider more general Fano group compactifications by using the ideas of \cite{LZZ}, in which the modified K-energy is discussed. Let $G$ be a connected, complex reductive group of dimension $n$,  we call $M$ a {\it (bi-equivariant) compactification of $G$} if it admits a holomorphic $G\times G$ action on $M$ with an open and dense orbit isomorphic to $G$ as a $G\times G$-homogeneous space \cite{AK, Del2}.  $(M, L)$ is called a {\it polarized compactification} of $G$ if  $L$ is a $G\times G$-linearized ample line bundle on $M$. In particular, when $L=-K_M$, we call $M$ a \emph{Fano group compactification}.
We establish the criterion for the existence of Mabuchi metrics on Fano group compactifications.

\begin{thm}\label{bar condition thm}
Let $(M,-K_M)$ be a Fano compactification of $G$ and $P$ be the associated polytope. Then  $M$ admits Mabuchi metrics
if and only if $c_X>0$ and
\begin{eqnarray}\label{bar condition}
\mathbf b_X-4\rho\in\Xi,
\end{eqnarray}
where
\begin{eqnarray*}
\mathbf b_X&=&{\frac1V}\int_{2P_+}y[1-\theta_X(y)]\pi(y)\,dy,\\
\pi(y)&=&\prod_{\alpha\in\Phi_+}\langle\alpha,y\rangle^2,\ \ V=\int_{2P_+}\pi(y)\,dy,
\end{eqnarray*}
$\Xi$ is the relative interior of the cone generated by positive roots $\Phi_+$, $\rho={\frac 1 2}\sum_{\alpha\in\Phi_+}\alpha$ and $\theta_X(y)$ is the normalized potential of $X$ viewed as a function on $2P_+$, which will be described in Lemma \ref{coefficients X} below. For notations on group compactifications, see \S\ref{grp cpt}.
\end{thm}

The paper is organized as follows:
In Section 2, we first review some preliminaries on energy functionals and
the definition of properness modulo an automorphism group. Then we recall basic properties of polarized compactifications. Theorem \ref{0301} and \ref{inverse proper} will be proved in Section 3. In  Section 4, we obtain Theorem \ref{bar condition thm}. The sufficient part will be proved by the verification of properness of the modifiend Ding  functional.

\vskip 5pt

{\bf Acknowledgments.} The authors would like to thank Professor Xiaohua Zhu and Yi Yao for many valuable discussions.


\section{Preliminaries}

In this section, we first review the  notions of energy functionals associated to Mabuchi metrics.
Then we recall the basic knowledge for group compactifications for later use.

\subsection{Reduction to the complex Monge-Amp\`ere equations}

It is clear that \eqref{mabuchi-def} is equivalent to the following equation
\begin{eqnarray}\label{0202}
\omega^n_\phi(1-\theta_X(\omega_\phi))=\omega^n_0e^{h_0-\phi}.
\end{eqnarray}
We consider the following continuity path
\begin{eqnarray}\label{0302}
\omega^n_{\phi_t}(1-\theta_X(\omega_{\phi_t}))=\omega^n_0e^{h_0-t\phi_t},~t\in[0,1].
\end{eqnarray}
Denote $\mathfrak I:=\{t\in[0,1]|(\ref{0302}) \text{ has a solution for $t$}\}$.
Then $\mathfrak I$ is open by the implicit function theorem. For the starting point $t=0$, we have

\begin{thm}\label{0501}
When $c_X>0$, (\ref{0302}) has a solution at $t=0$.
\end{thm}

Since we did not find a reference for this result, we give a proof of it for completeness in the appendix.
Hence, $0\in\mathfrak I$ and there exists an $\epsilon_0>0$ such that (\ref{0302}) has a solution for $t\in[0,\epsilon_0]$. For the closedness of $\mathfrak I$,
it suffices to establish the $C^0$-estimate of (\ref{0302}).
The following lemmas will be used later.
\begin{lem}\label{0302+}
Let $\phi_t$ be a solution of (\ref{0302}) at $t$, then the first eigenvalue of
\begin{eqnarray}\label{eigenvalue}
L_t:=\bigtriangleup_{\omega_{\phi_t}}-{\frac{X}{1-\theta_X(\omega_{\phi_t})}}+t
\end{eqnarray}
is nonnegative for $t\in[0,1]$ and equals to $0$ only if $t=1$. Consequently, we have the following weighted Poincar\'e inequality
\begin{eqnarray}\label{Poincare}
&&\int_M|\bar\partial\psi|_{\omega_{\phi_t}}^2(1-\theta_X(\omega_{\phi_t}))\omega_{\phi_t}^n\notag\\
&\geq& t\left[\int_M\psi^2(1-\theta_X(\omega_{\phi_t}))\omega^n_{\phi_t}-{\frac1V}\left(\int_M\psi(1-\theta_X(\omega_{\phi_t}))\omega_{\phi_t}^n\right)^2\right].
\end{eqnarray}
for any $K_X$-invariant $\psi\in C^{1,\alpha}$.
\end{lem}
\begin{rem}\label{rmk 2.3}
We remark that $L_t$ is self-dual on the space of real-valued $K_X$-invariant functions, equipped with the weighted inner product (cf. \cite[Lemma 2.1]{Mab2})
\begin{eqnarray*}
\langle f,g\rangle_t:=\int_MfL_t(g)(1-\theta_X(\omega_{\phi_t}))\omega_{\phi_t}^n.
\end{eqnarray*}
\end{rem}
\begin{proof}[Proof of Lemma \ref{0302+}]
Without loss of generality, we may choose a local co-frame $\{\Theta^i\}_{i=1}^n$ such that $\omega_{\phi_t}=\sqrt{-1}\sum_{i=1}^n\Theta^i\wedge\bar\Theta^i$. Suppose
$L_t\psi=-\lambda\psi$.
Then
\begin{eqnarray}\label{0205}
&&\lambda\int_M\psi_{,i}\psi_{,\bar i}(1-\theta_X(\omega_{\phi_t}))\omega^n_{\phi_t}\notag\\
&=&-\int_M\left[\left(\bigtriangleup_{\omega_{\phi_t}}-{\frac{X}{1-\theta_X(\omega_{\phi_t})}}+t\right)\psi\right]_{,i}\psi^{,i}(1-\theta_X(\omega_{\phi_t}))\omega^n_{\phi_t}\notag\\
&=&-\int_M{{\psi_{,j\bar ji}}}\psi_{,\bar i}(1-\theta_X(\omega_{\phi_t}))\omega^n_{\phi_t}+\int_MX^j_{,i}\psi_{,\bar i}\psi_{,j}\omega^n_{\phi_t}
+\int_MX^i\psi_{,ij}\psi_{,\bar j}\omega^n_{\phi_t}\notag\\
&&+\int_M{\frac{X(\psi)\theta_X(\omega_{\psi_t})_{,i}}{[1-\theta_X(\omega_{\phi_t})]^2}}\psi_{,\bar i}(1-\theta_X(\omega_{\phi_t}))\omega^n_{\phi_t}-t\int_M\psi_{,i}\psi_{,\bar i}(1-\theta_X(\omega_{\phi_t}))\omega^n_{\phi_t},
\end{eqnarray}
here and below, we denote $\phi_{,i}$ for covariant derivatives with respect to $\omega_{\phi_t}$, similar conventions are used for covariant derivatives of other tensors.

By Ricci identity and integration by parts, we have
\begin{eqnarray*}\label{0206}
&&-\int_M{{\psi_{,j\bar ji}}}\psi_{,\bar i}(1-\theta_X(\omega_{\phi_t}))\omega^n_{\phi_t}\notag\\
&=&-\int_M{{\psi_{,ij\bar j}}}\psi_{,\bar i}(1-\theta_X(\omega_{\phi_t}))\omega^n_{\phi_t}+\int_M\psi_{,\bar j}\psi_{,\bar i}Ric_{i\bar j}(1-\theta_X(\omega_{\phi_t}))\omega^n_{\phi_t}\notag\\
&=&\int_M{\psi_{,ij}}\psi_{,\bar i\bar j}(1-\theta_X(\omega_{\phi_t}))\omega^n_{\phi_t}-\int_MX^i\psi_{,ij}\psi_{,\bar j}\omega^n_{\phi_t}+\int_M\psi_{,\bar j}\psi_{,\bar i}Ric_{i\bar j}(1-\theta_X(\omega_{\phi_t}))\omega^n_{\phi_t}.
\end{eqnarray*}
Substituting this into (\ref{0205}) and using (\ref{0302}), it follows
\begin{eqnarray*}
&&\lambda\int_M\psi_{,i}\psi_{,\bar i}(1-\theta_X(\omega_{\phi_t}))\omega^n_{\phi_t}\notag\\
&=&\int_M{\psi_{,ij}}\psi_{,\bar i\bar j}(1-\theta_X(\omega_{\phi_t}))\omega^n_{\phi_t}+(1-t)\int_M\psi_{,\bar i}\psi_{,j}g_{i\bar j}(0)(1-\theta_X(\omega_{\phi_t}))\omega^n_{\phi_t},
\end{eqnarray*}
where $\omega_0=\sqrt{-1}g_{i\bar j}(0)\Theta^i\wedge\bar\Theta^j$, the lemma is proved.
\end{proof}


\subsection{Energy functionals}

Recall that the Aubin's functionals are given by
\begin{eqnarray*}
I(\phi)&=&\int_M\phi(\omega_0^n-\omega_\phi^n),\\
J(\phi)&=&\int_0^1\int_M\dot\phi_s(\omega_0^n-\omega_{\phi_s}^n)\wedge ds,
\end{eqnarray*}
where $\{\phi_s\}_{s\in[0,1]}$ is any smooth path in $\mathcal H_{X}(\omega_0)$ joining $0$ and $\phi$. It is known that \cite{Tian297}
\begin{eqnarray}\label{0204+}
0\leq {\frac1n}J(\phi)\leq I(\phi)-J(\phi) \leq nJ(\phi).
\end{eqnarray}
To deal with Mabuchi metrics, the following modified functionals  were introduced in \cite{Mab2}
\begin{eqnarray*}
I_{X}
(\phi)&=&\int_M\phi[(1-\theta_X(\omega_0))\omega_0^n-(1-\theta_X(\omega_\phi))\omega_\phi^n],\\
J_{X}(\phi)&=&\int_0^1\int_M\dot\phi_s[(1-\theta_X(\omega_0))\omega_0^n-(1-\theta_X(\omega_{\phi_s}))\omega_{\phi_s}^n]\wedge ds.
\end{eqnarray*}
By \cite[Remark A.1.9]{Mab2}, when $c_X>0$,
\begin{eqnarray}\label{0204}
0\leq I_{X}(\phi)\leq(n+2)(I_{X}(\phi)-J_{X}(\phi))\leq(n+1)I_{X}(\phi).
\end{eqnarray}

\begin{lem}\label{0303}
There are positive constants $c_1,c_2>0$ such that
\begin{eqnarray}\label{0304}
c_1I(\phi)\leq I_{X}(\phi)-J_{X}(\phi)\leq c_2I(\phi).
\end{eqnarray}
\end{lem}
\begin{proof}
Take a path $\phi_s=s\phi$. Then
\begin{eqnarray*}
{\frac{d}{ds}}[I_{X}(\phi_s)-J_{X}(\phi_s)]&=&-s\int_M\phi\cdot\left(\bigtriangleup_{\omega_{\phi_s}}-{\frac{X}{1-\theta_X(\omega_{\phi_s})}}\right)\phi\cdot (1-\theta_X(\omega_{\phi_s}))\omega_{\phi_s}^n\\
&=&s\int_M|\partial\phi|^2_{\omega_{\phi_s}}(1-\theta_X(\omega_{\phi_s}))\omega_{\phi_s}^n,
\end{eqnarray*}
Note that
\begin{eqnarray*}
{\frac{d}{ds}}[I(\phi_s)-J(\phi_s)]=s\int_M|\partial\phi|^2_{\omega_{\phi_s}}\omega_{\phi_s}^n.
\end{eqnarray*}
When $c_X>0$, it follows
\begin{eqnarray*}
0\leq c_X{\frac{d}{ds}}[I(\phi_s)-J(\phi_s)]\leq{\frac{d}{ds}}[I_{X}(\phi_s)-J_{X}(\phi_s)]\leq C_X{\frac{d}{ds}}[I(\phi_s)-J(\phi_s)].
\end{eqnarray*}
Thus the lemma follows from (\ref{0204+}).
\end{proof}

In view of \cite{CTZ, DR, Tian97, Tian, ZZ}, we have the following definition of properness:
\begin{definition}\label{prop def}
Suppose $H^c$ is a reductive subgroup  (which is the complexfication of a compact Lie group $H$) of $Aut_0(M)$ which contains $K_X$.
The modified Ding functional $\mathcal D_{X}(\cdot)$ is said to be  \emph{proper modulo $H^c$}
if there exists an increasing function $f(t)\geq-c$ for $t\in\mathbb R$ and some constant $c\geq0$ such that $\displaystyle\lim_{t\to+\infty}f(t)=+ \infty$ and
\begin{eqnarray*}
\mathcal D_{X}(\phi)\geq \inf_{\sigma\in H^c}f(I_{X}(\phi_\sigma)-J_{X}(\phi_\sigma)),
\end{eqnarray*}
where $\phi_\sigma$ is defined by
$\sigma^*(\omega_\phi)=\omega_0+\sqrt{-1}\partial\bar\partial\phi_\sigma$.
\end{definition}

For convenience, we write the modified Ding functional \eqref{0203} as $\mathcal D_{X}(\phi)=\mathcal N(\phi)+\mathcal D^0_X(\phi)$, where
\begin{eqnarray}
\mathcal N(\phi)&=& -\log\left({\frac1V}\int_Me^{h_0-\phi}\omega_0^n\right),\label{linear}\\
\mathcal D^0_X(\phi)&=&-{\frac1V}\int_0^1\int_M\dot\phi_s(1-\theta_X(\omega_{\phi_s}))\omega_{\phi_s}^n\wedge ds. \label{nonlinear}
\end{eqnarray}
It is known that   $\mathcal N$ is convex with respect to geodesics \cite{Bern}.
In the latter proof of Theorem \ref{inverse proper}, we need the convexity  of $\mathcal D^0_X(\cdot)$.

\begin{lem}\label{linear properties}
The functional $\mathcal D^0_{X}(\cdot)$ satisfies:
\begin{itemize}
\item[(1)] When $c_X>0$, $\mathcal D^0_{X}(\cdot)$ is monotonic, that is for any $\phi_0\leq\phi_1$, $\mathcal D^0_{X}(\phi_0)\geq\mathcal D^0_{X}(\phi_1)$;

\item[(2)] $\mathcal D^0_{X}(\cdot)$ is affine along any $C^{1,1}$-geodesic connecting two smooth potentials in $\mathcal H_X(\omega_0)$.
\end{itemize}
\end{lem}
\begin{proof}
To see (1), by definition we have
\begin{eqnarray*}
\mathcal D^0_{X}(\phi_1)=\mathcal D^0_{X}(\phi_0)-{\frac1V}\int_0^1\int_M\dot\phi_s(1-\theta_X(\omega_{\phi_s}))\omega_{\phi_s}^n\wedge ds,
\end{eqnarray*}
where $\phi_s$ is any smooth path in $\mathcal H_X(\omega_0)$ joining $\phi_0$ and $\phi_1$. Take in particular $\phi_s=s(\phi_1-\phi_0)+\phi_0$ and note that $c_X>0$, we have
\begin{eqnarray*}
\mathcal D^0_{X}(\phi_1)=\mathcal D^0_{X}(\phi_0)-{\frac1V}\int_0^1\int_M(\phi_1-\phi_0)(1-\theta_X(\omega_{\phi_s}))\omega_{\phi_s}^n\wedge ds\leq\mathcal D^0_{X}(\phi_0).
\end{eqnarray*}

Next we prove (2).
Let $\{\phi_t\}$ be the $C^{1,1}$-geodesic connecting  $\phi_0,\phi_1\in\mathcal H_X(\omega_0)$.
By \cite{Chen}, $\{\phi_t\}$ can be approximated by a family of smooth $\epsilon$-geodesic $\{\phi_t^\epsilon| t\in\Omega\}$ in $\mathcal H_X(\omega_0)$ connecting $\phi_0$ and $\phi_1$, satisfying
\begin{eqnarray}\label{eps-sub-geod}
\left({\frac{\partial^2}{\partial \tau\partial \bar \tau}}\phi_t^\epsilon-|\bar\partial\phi_t^\epsilon|_{\omega_{\phi_t}}^2\right)(\omega_0+\sqrt{-1}\partial\bar\partial \phi_t^{\epsilon})^n=\epsilon\cdot \omega_0^n,
\end{eqnarray}
on $M\times \Omega$, where $\Omega:=[0,1]\times S^1\in \mathbb C$ and $t=Re(\tau)$.
For each $\epsilon$, we have
\begin{eqnarray*}
{\frac{\partial}{\partial \tau}}\mathcal D^0_{X}(\phi_t^\epsilon)=-{\frac1V}\int_M{\frac{\partial}{\partial \tau}}\phi_t^\epsilon(1-\theta_X(\omega_{\phi_t^\epsilon}))\omega_{\phi_t^\epsilon}^n.
\end{eqnarray*}
It follows
\begin{eqnarray}\label{0212}
{\frac{\partial^2}{\partial \tau\partial \bar \tau}}\mathcal D^0_{X}(\phi_t^\epsilon)&=&-{\frac1V}\int_M{\frac{\partial^2\phi_t^\epsilon}{\partial \tau\partial \bar \tau}}(1-\theta_X(\omega_{\phi_t^\epsilon}))\omega_{\phi_t^\epsilon}^n+{\frac1V}\int_M{\frac{\partial \phi_t^\epsilon}{\partial \tau}}{\frac{\partial \theta_X(\omega_{\phi_t^\epsilon})}{\partial \bar \tau}}\omega_{\phi_t^\epsilon}^n\notag\\
&&-{\frac{\sqrt{-1}}V}\int_M{\frac{\partial \phi_t^\epsilon}{\partial \tau}}(1-\theta_X(\omega_{\phi_t^\epsilon}))n\omega_{\phi_t^\epsilon}^{n-1}\wedge\partial\bar\partial{\frac{\partial \phi_t^\epsilon}{\partial \bar \tau}}.
\end{eqnarray}
Recall that $\theta_X(\omega_{\phi_t^\epsilon})=\theta_X(\omega_{0})+X(\phi_t^\epsilon)$, one gets
\begin{eqnarray*}
{\frac1V}\int_M{\frac{\partial \phi_t^\epsilon}{\partial \tau}}{\frac{\partial \theta_X(\omega_{\phi_t^\epsilon})}{\partial \bar \tau}}\omega_{\phi_t^\epsilon}^n={\frac1V}\int_M{\frac{\partial \phi_t^\epsilon}{\partial \tau}} X^i\left({\frac{\partial \phi_{t}^\epsilon}{\partial \bar \tau}}\right)_{,i}\omega_{\phi_t^\epsilon}^n.
\end{eqnarray*}
On the other hand, by integration by parts, we have
\begin{eqnarray*}
&&{\frac{\sqrt{-1}}V}\int_M{\frac{\partial \phi_t^\epsilon}{\partial \tau}}(1-\theta_X(\omega_{\phi_t^\epsilon}))n\omega_{\phi_t^\epsilon}^{n-1}\wedge\partial\bar\partial{\frac{\partial \phi_t^\epsilon}{\partial \bar \tau}}\\
&=&{\frac{\sqrt{-1}}V}\left[\int_M \bar\partial {\frac{\partial \phi_t^\epsilon}{\partial \tau}}(1-\theta_X(\omega_{\phi_t^\epsilon}))n\omega_{\phi_t^\epsilon}^{n-1}\wedge
\partial{\frac{\partial \phi_t^\epsilon}{\partial \bar \tau}}
-\int_M{\frac{\partial \phi_t^\epsilon}{\partial \tau}}\bar\partial\theta_X(\omega_{\phi_t^\epsilon})n\omega_{\phi_t^\epsilon}^{n-1}\wedge\partial{\frac{\partial \phi_t^\epsilon}{\partial \bar \tau}}\right]\\
&=&-{\frac{1}V}\int_M\left|{\frac{\partial \phi_t^\epsilon}{\partial \tau}}\bar\partial\phi_t^\epsilon\right|^2_{\omega_{\phi_t}}(1-\theta_X(\omega_{\phi_t^\epsilon}))\omega_{\phi_t^\epsilon}^n
+{\frac1V}\int_M{\frac{\partial\phi_t^\epsilon}{\partial \tau}} X^i\left({\frac{\partial \phi_{t}^\epsilon}{\partial \bar \tau}}\right)_{,i}\omega_{\phi_t^\epsilon}^n.
\end{eqnarray*}
Plugging these into (\ref{0212}), by (\ref{eps-sub-geod}), we have
\begin{eqnarray*}
{\frac{\partial^2}{\partial \tau\partial \bar \tau}}\mathcal D^0_{X}(\phi_t^\epsilon)=-\epsilon<0.
\end{eqnarray*}
Thus $\mathcal D^0_{X}(\cdot)$ is concave along $\phi_t^\epsilon$.
Sending $\epsilon\to0$,  $\sqrt{-1}\partial\bar\partial_\tau \mathcal D^0_{X}(\phi_t^\epsilon)$ converges weakly to $\sqrt{-1}\partial\bar\partial_\tau \mathcal D^0_{X}(\phi_t)$ as Monge-Amp\'ere measures.
 It follows $\sqrt{-1}\partial\bar\partial_\tau \mathcal D^0_{X}(\phi_t)=0$, thus $\mathcal D^0_{X}(\phi_t)$ is affine as desired.
\end{proof}

\subsection{Group compactifications}\label{grp cpt}
As an application of Theorem \ref{0301}, we will study the existence of Mabuchi metrics on group compactifications by testing properness of the modified Ding functional. The existence of K\"ahler-Einstein metrics on these manifolds has been solved by \cite{Del2} by using the continuity method, while the properness of K-energy was studied in \cite{LZZ}. We will prove Theorem \ref{bar condition thm} by ideas therein later. In this subsection, we recall some facts of group compatifications from \cite{Del2, LZZ}.

\subsubsection{Notations on Lie groups}
Choose a maximal compact subgroup $K$ of $G$ such that $G$ is its complexification. Let $T$ be a chosen maximal torus of $K$ and $T^c$ be its complexification, then $T^c$ is the maximal complex torus of $G$. Denote their Lie algebras by the corresponding fraktur lower case letters. Assume that $\Phi$ is the root system of $(G,T^\mathbb C)$ and $W$ is the Weyl group. Choose a set of positive roots $\Phi_+$. Set $\rho={\frac 1 2}\sum_{\alpha\in\Phi_+}\alpha$ and $\Xi$ be the relative interior of the cone generated by $\Phi_+$.
Let $J$ be the complex structure of $G$, then $$\mathfrak g=\mathfrak k\oplus J\mathfrak k.$$
Set $\mathfrak a=J\mathfrak t$, it can be decomposed as a toric part and a semisimple part:
$$\mathfrak a = \mathfrak a_t\oplus\mathfrak a_{ss},$$
where $\mathfrak a_t:= \mathfrak z(\mathfrak g)\cap\mathfrak a$ and $\mathfrak a_{ss}:= \mathfrak a\cap[\mathfrak g,\mathfrak g]$. We extend the Killing form on $\mathfrak a_{ss}$ to a scalar product $\langle\cdot,\cdot\rangle$ on $\mathfrak a$ such that $\mathfrak a_t$ is orthogonal to $\mathfrak a_{ss}$. 
The positive roots $\Phi_+$ defines a positive Weyl chamber $\mathfrak a_+\subset\mathfrak a$, and a positive Weyl chamber ${\mathfrak a}^*_+$ on $\mathfrak a^*$,  where
$${\mathfrak a}^*_+:=\{y|~\alpha(y):=\langle\alpha,y\rangle>0,~\forall\alpha\in\Phi_+\},$$
it coincides with the dual of $\mathfrak a_+$ under $\langle\cdot,\cdot\rangle$.
For later use, we fix a Lebesgue measure $dy$ on $\mathfrak a^*$ which is normalized by the lattice of the characters of $T^c$.

\subsubsection{$K\times K$-invariant K\"ahler metrics}

Let $Z$ be the closure of $T^c$ in $M$. It is known that $(Z, L|_Z)$ is a polarized toric manifold with a $W$-action, and $L|_Z$ is a $W$-linearized ample toric line bundle on $Z$  \cite{AB1,AB2, AK, Del2}. Let $\omega_0\in 2\pi c_1(L)$ be  a $K\times K$-invariant K\"ahler form induced from $(M, L)$  and  $P$ be the polytope associated to $(Z, L|_Z)$, which is defined by the moment map associated to $\omega_0$.   Then $P$ is a $W$-invariant  delzent polytope in $\mathfrak a^*$.
By the $K\times K$-invariance,  for any
$$\phi\in\mathcal H_{K\times K}(\omega_0):=\{\phi\in C^{\infty}(M)|\omega_{\phi}>0,~\phi\text{ is $K\times K$-invariant}\},$$
the restriction of $\omega_\phi$ on $Z$ is a toric K\"ahler metric. It induces a smooth strictly convex function $\psi_\phi$ on ${\mathfrak a}$, which is $W$-invariant \cite{AL, Del2}.

By the $KAK$-decomposition (\cite{Kna}, Theorem 7.39), for any $g\in G$,
there are $k_1,\,k_2\in K$ and $x\in\mathfrak a$ such that $g=k_1\exp(x)k_2$. Here $x$ is uniquely determined up to a $W$-action. This means that $x$ is unique in $\bar{\mathfrak a}_+$.
Thus there is a bijection between smooth $K\times K$-invariant functions  $\Psi$   on $G$ and smooth $W$-invariant functions on $\mathfrak a$ which is given  by
$$\Psi( \exp(\cdot))=\psi(\cdot):~{\mathfrak a}\to\mathbb R.$$
Clearly when a $W$-invariant $\psi$ is given, $\Psi$ is well-defined. 
In the following, we will not distinguish $\psi$ and $\Psi$. 
The following $KAK$-integral formula can be found in \cite[Proposition 5.28]{Kna2} (see also \cite{HY}).

\begin{prop}\label{KAK int}
Let $dV_G$ be a Haar measure on $G$ and $dx$  the Lebesgue measure
on $\mathfrak{a}$.
Then there exists a constant $C_H>0$ such that for any
$K\times K$-invariant, $dV_G$-integrable function $\psi$ on $G$,
$$\int_G \psi(g)\,dV_G= C_H\int_{\mathfrak{a}_+}\psi(x)\mathbf{J}(x)\,dx,$$
where
$\mathbf J(x)=\prod_{\alpha \in \Phi_+} \sinh^2\alpha(x).$
\end{prop}
With out loss of generality, we can normalize $C_H=1$ for simplicity.

Next we recall a local holomorphic coordinates on $G$ used in  \cite{Del2}.   By the standard Cartan decomposition, we can decompose $\mathfrak g$ as
$$\mathfrak g=\left(\mathfrak t\oplus\mathfrak a\right)\oplus\left(\oplus_{\alpha\in\Phi}V_{\alpha}\right),$$
where $V_{\alpha}=\{X\in\mathfrak g|~ad_H(X)=\alpha(H)X,~\forall H\in\mathfrak t\oplus\mathfrak a\}$, the root space of complex dimension $1$ with respect to $\alpha$. By \cite{Hel}, one can choose $X_{\alpha}\in V_{\alpha}$ such that $X_{-\alpha}=-\iota(X_{\alpha})$ and
$[X_{\alpha},X_{-\alpha}]=\alpha^{\vee},$ where $\iota$ is the Cartan involution and $\alpha^{\vee}$ is the dual of $\alpha$ by the Killing form.
Let $E_{\alpha}:=X_{\alpha}-X_{-\alpha}$ and $E_{-\alpha}:=J(X_{\alpha}+X_{-\alpha})$. Denote by $\mathfrak k_{\alpha},\,\mathfrak k_{-\alpha}$ the real line spanned by $E_\alpha,\,E_{-\alpha}$, respectively.
Then we have the Cartan decomposition of $\mathfrak k$,
$$\mathfrak k=\mathfrak t\oplus\left(\oplus_{\alpha\in\Phi_+}\left(\mathfrak k_{\alpha}\oplus\mathfrak k_{-\alpha}\right)\right).$$
Denote by $r$ the dimension of $T$, choose a real basis $\{E^0_1,...,E^0_r\}$ of $\mathfrak t$.  Then $\{E^0_1,...,E^0_r\}$ together with $\{E_{\alpha},E_{-\alpha}\}_{\alpha\in\Phi_+}$ forms a real basis of $\mathfrak k$, which is indexed by $\{E_1,...,E_n\}$. $\{E_1,...,E_n\}$ can also be regarded as a complex basis of $\mathfrak g$. For any $g\in G$, we define local coordinates $\{z_{(g)}^i\}_{i=1,...,n}$ on a neighborhood of $g$ by
$$(z_{(g)}^i)\to\exp(z_{(g)}^iE_i)g.$$
It is easy to see  that $\theta^i|_g=dz_{(g)}^i|_g$,  where $\theta^i$ is the  dual of $E_i$,  which is a right-invariant holomorphic  $1$-form.    Thus
$\displaystyle{\wedge_{i=1}^n\left(dz_{(g)}^i\wedge d\bar{z}_{(g)}^i\right)}|_g$ is  also a  right-invariant  $(n,n)$-form,  which defines a Haar measure  $dV_G$.

The derivations of the $K\times K$-invariant function  $\psi$ in the above local coordinates was computed by  Delcroix as follows  \cite[Theorem 1.2]{Del2}.

\begin{lem}\label{Derivative}
Let $\psi$ be a $K\times K$ invariant function on $G$. 
Then for any $x\in \mathfrak{a}_+$,
\begin{eqnarray*}
E_i^0(\psi)|_{\exp(x)}&=&d\psi(\text{Im}(E_i^0))|_x,~1\leq i\leq r,\\
E_{\pm\alpha}(\psi)|_{\exp(x)}&=&0.
\end{eqnarray*}
\end{lem}

\begin{lem}\label{Hessian}
Let $\psi$ be a $K\times K$ invariant function on $G$, 
then  for any $x\in \mathfrak{a}_+$,
the complex Hessian matrix of $\psi$  in the  above coordinates is diagonal by blocks, and equals to
\begin{equation}\label{+21}
\mathrm{Hess}_{\mathbb{C}}(\psi)(\exp(x)) =
\begin{pmatrix}
\frac{1}{4}\mathrm{Hess}_{\mathbb{R}}(\psi)(x)& 0 &  & & 0 \\
 0 & M_{\alpha_{(1)}}(x) & & & 0 \\
 0 & 0 & \ddots & & \vdots \\
\vdots & \vdots & & \ddots & 0\\
 0 & 0 &  & & M_{\alpha_{(\frac{n-r}{2})}}(x)\\
\end{pmatrix},
\end{equation}
where $\Phi_+=\{\alpha_{(1)},...,\alpha_{(\frac{n-r}{2})}\}$ is the set of positive roots and
\[
M_{\alpha_{(i)}}(x) = \frac{1}{2}\langle\alpha_{(i)},\nabla \psi(x)\rangle
\begin{pmatrix}
\coth\alpha_{(i)}(x) & \sqrt{-1} \\
-\sqrt{-1} & \coth\alpha_{(i)}(x) \\
\end{pmatrix}.
\]
\end{lem}
By (\ref{+21}) in   Lemma \ref{Hessian},
we see that a $\psi$ induced by some $\omega_\phi$ is convex on $\mathfrak{a}$.
The complex Monge-Amp\'ere measure is given by $\omega^n=(\sqrt{-1}\partial\bar{\partial}\psi_\phi)^n=\mathrm{MA}_{\mathbb C}(\psi_\phi)\,dV_G$, where
\begin{equation}\label{MA}
\mathrm{MA}_{\mathbb{C}}(\psi_\phi)(\exp(x))= \frac{1}{2^{r+n}}
\mathrm{MA}_{\mathbb{R}}(\psi_\phi)(x)\frac{1}{\mathbf J(x)}\prod_{\alpha \in \Phi_+}\langle\alpha,\nabla \psi_\phi(x)\rangle^2.
\end{equation}
\subsubsection{Legendre functions}
By the  convexity  of $\psi_\phi$ on $\mathfrak a$,   the gradient $\nabla \psi_\phi$  defines a diffeomorphism from $\mathfrak{a}$ to the interior of the dilated polytope $2P$\footnote{We remark that the moment map is given by ${\frac12}\nabla \psi_\phi$, whose image is $P$.}.
Let $P_+:=P\cap\bar{\mathfrak a}_+^*$.  Then by the $W$-invariance of $\psi_\phi$ and $P$, the restriction of $\nabla \psi_\phi$ on $\mathfrak{a}_+$ is a diffeomorphism from $\mathfrak{a}_+$ to the interior of $2P_+$. Let $u_G$ be the standard Guillemin function on $2P$ \cite{Gui}. Set
\begin{equation}\label{L}
\begin{aligned}
\mathcal C_{W}=\{u|~u \text{ is strictly convex, } u-u_G\in C^{\infty}(\overline{2P}) \text{ and } u \text{ is } W\text{-invariant}\}.
\end{aligned}\nonumber
\end{equation}
It is known that for any $K\times K$-invariant $\omega=\sqrt{-1}\partial\bar\partial\psi\in2\pi c_1(L)$, its Legendre function $u$ is given by
\begin{eqnarray}\label{Legendre}
u(y(x))=x^iy_{i}(x)-\psi(x), \   y_{i}(x)=\psi_{i}(x)=\frac{\partial \psi}{\partial y_i}
\end{eqnarray}
is a function in $\mathcal C_{W}$ (cf. \cite{Ab}).
By a similar argument as \cite{Gua} for toric manifolds, we have
\begin{lem}\label{0601}
For any $\phi_0, \phi_1\in \mathcal H_{K\times K}(\omega_0)$, there exists  a geodesic  $\{\phi_t\}_{t\in[0,1]}$ in $\mathcal H_{K\times K}(\omega_0)$ joining them, and the Legendre function of $\psi_\phi$ is given by
$u_{\phi_t}=(1-t)u_{\phi_0}+tu_{\phi_1}$.
\end{lem}

\section{Proof of the properness theorem}

Theorem \ref{0301} will be proved by steps as for K\"ahler-Ricci solitons \cite{CTZ, TianZhu Acta}.
We always assume $c_X>0$ in this section.

First, we have
\begin{lem}\label{I bound}
Let $\phi_t$ be a solution of (\ref{0302}) at $t$, if $I_X(\phi_t)$ is uniformly bounded, then there is a uniform constant $C$ such that
$$|\phi_t|\leq C,~\forall t\in[0,1].$$
\end{lem}
\begin{proof}
This estimate was essentially obtained in \cite{Mab2}. Here we will give a different proof following the arguments of \cite{TianZhu Acta}. In view of Ko{\l}odziej's $L^\infty$-estimate \cite{Kolodziej} for complex Monge-Amp\`ere equation, it suffices to obtain the $L^p$-estimate of $e^{-t\phi_t}$ for some $p>1$.

By the assumption,
$0\leq I_X(\phi_t)\leq C_1$ for some uniform $C_1$.
By (\ref{0302}), we have
\begin{eqnarray*}
\int_Me^{h_0-t\phi_t}\omega_0^n=\int_M(1-\theta_X(\omega_{\phi_t}))\omega_{\phi_t}^n=\int_Me^{h_0}\omega_0^n,
\end{eqnarray*}
thus
\begin{eqnarray*}
\inf_M\phi_t\leq0\leq\sup_M\phi_t.
\end{eqnarray*}
While by (\ref{0302}),
\begin{eqnarray*}
-t\int_M\phi_t\omega_{\phi_t}^n=-t\int_M\phi_t\frac{e^{h_0-t\phi_t}}{1-\theta_X(\omega_{\phi_t})}\omega_{0}^n\geq -C_2t\int_{\{\phi_t\geq0\}}\phi_te^{-t\phi_t}\omega_{0}^n\geq -C_3.
\end{eqnarray*}
Thus
\begin{eqnarray}\label{0207}
t\int_M\phi_t\omega_{0}^n\leq C_4.
\end{eqnarray}
Let $\Gamma(\cdot,\cdot)$ be the Green function of $\omega_0$. Then by $\bigtriangleup_{\omega_0}\phi_t>-n$,  $\Gamma+C_{\Gamma}\geq0$ for some $C_\Gamma>0$.
By (\ref{0207}) and Green's formula, we have
\begin{eqnarray}\label{0208}
t\sup_M\phi_t\leq{\frac  t V}\int_M\phi_t\omega_0^n-{\frac tV}\min_{M}\left(\int_M(\Gamma(x,\cdot)+C_\Gamma)\bigtriangleup_{\omega_0}\phi_t\omega_0^n\right)\leq C_5.
\end{eqnarray}
By the boundness of $I_X(\phi_t)$, we have
\begin{eqnarray}
-{\frac1V}\int_M\phi_t\omega_{\phi_t}^n\leq C_1-{\frac1V}\int_M\phi_t\omega_0^n\leq C_6.
\end{eqnarray}
Moreover,
\begin{eqnarray}\label{0209+}
-t\int_{\{\phi\leq0\}}\phi_t\omega_{\phi_t}^n&=& -t\int_M\phi_t\omega_{\phi_t}^n+t\int_{\{\phi_t\geq0\}}\phi_t\omega_{\phi_t}^n\notag\\
&\leq&tVC_6+t\int_{\{\phi_t\geq0\}}\phi_t\frac{e^{h_0-t\phi_t}}{1-\theta_X(\omega_{\phi_t})}\omega_{0}^n\leq C_7.
\end{eqnarray}
By (\ref{0208}), there is a uniform $C>0$ such that
$\hat\phi_t:=\phi_t-{\frac Ct}\leq-1$.
By (\ref{0209+}), it follows
\begin{eqnarray*}
-t\int_M\phi_t(1-\theta_X(\omega_{\phi_t}))\omega_{\phi_t}^n
\leq -t\int_{\{\phi_t\leq0\}}\phi_t(1-\theta_X(\omega_{\phi_t}))\omega_{\phi_t}^n\leq C_XC_4,
\end{eqnarray*}
and consequently,
\begin{eqnarray}\label{0211}
-t\int_M\hat\phi_t(1-\theta_X(\omega_{\phi_t}))\omega_{\phi_t}^n\leq C_8.
\end{eqnarray}

On the other hand,
\begin{eqnarray*}
\int_M\left|\bar\partial(-\hat\phi_t)^{\frac{p+1}2}\right|_{\omega_{\phi_t}}^2\omega_{\phi_t}^n&=&\frac{n(p+1)^2}{4p}\int_M(-\hat\phi_t)^p(\omega_{\phi_t}^n-\omega_{\phi_t}^{n-1}\wedge\omega_0)\notag\\
&\leq&\frac{n(p+1)^2}{4p}\int_M(-\hat\phi_t)^p\omega_{\phi_t}^n.
\end{eqnarray*}
Recall that $0<c_X<1-\theta_X(\omega_{\phi_t})<C_X$.
Combining the above inequality with (\ref{Poincare}),
\begin{eqnarray*}
&&\int_M(-\hat\phi_t)^{p+1}(1-\theta_X(\omega_{\phi_t}))\omega^n_{\phi_t}\notag\\
&\leq&\frac{Cp}{t}\int_M(-\hat\phi_t)^{p}(1-\theta_X(\omega_{\phi_t}))\omega^n_{\phi_t}+{\frac1V}\left(\int_M(-\hat\phi_t)^{\frac{p+1}{2}}(1-\theta_X(\omega_{\phi_t}))\omega^n_{\phi_t}\right)^2\notag\\
&\leq&\frac{Cp}{t}\int_M(-\hat\phi_t)^{p}(1-\theta_X(\omega_{\phi_t}))\omega^n_{\phi_t}\notag\\
&&+{\frac1V}\left(\int_M(-\hat\phi_t)^{p}(1-\theta_X(\omega_{\phi_t}))\omega^n_{\phi_t}\right)\left(\int_M(-\hat\phi_t)(1-\theta_X(\omega_{\phi_t}))\omega^n_{\phi_t}\right)\\
&\leq&\frac{{C'}p}{t}\int_M(-\hat\phi_t)^{p}(1-\theta_X(\omega_{\phi_t}))\omega^n_{\phi_t},
\end{eqnarray*}
where we used (\ref{0211}) in the last line.
By iteration and using (\ref{0211}), we have
\begin{eqnarray*}
\int_M(-\hat\phi_t)^{p+1}(1-\theta_X(\omega_{\phi_t}))\omega^n_{\phi_t}\leq{\frac{C'^p(p+1)!}{t^p}}\int_M(-\hat\phi_t)(1-\theta_X(\omega_{\phi_t}))\omega^n_{\phi_t}\leq{\frac{C^{p+1}(p+1)!}{t^{p+1}}}.
\end{eqnarray*}
Thus for $0<\epsilon<1/c$,
\begin{eqnarray*}
\int_Me^{-t\epsilon\hat\phi_t}(1-\theta_X(\omega_{\phi_t}))\omega^n_{\phi_t}=\sum_{p=0}^{+\infty}{\frac{(t\epsilon)^p}{p!}}\int_M(-\hat\phi_t)^{p}(1-\theta_X(\omega_{\phi_t}))\omega^n_{\phi_t}
\leq{\frac1{1-c\epsilon}}.
\end{eqnarray*}
It follows
\begin{eqnarray*}
\int_Me^{-t(1+\epsilon)\phi_t}\omega^n_{0}&=&\int_Me^{-t(1+\epsilon)\phi_t}e^{-h_0-t\phi_t}(1-\theta_X(\omega_{\phi_t}))\omega^n_{\phi_t}\notag\\
&\leq&C_9\int_Me^{-t\epsilon\hat\phi_t}(1-\theta_X(\omega_{\phi_t}))\omega^n_{\phi_t}\leq C.
\end{eqnarray*}
Then the lemma then follows from Ko{\l}odziej's result.
\end{proof}

The following lemma was proved in \cite{Li}.

\begin{lem}\label{0305}
Fix $\epsilon_0\in (0,1)$. Then the modified Ding functional $\mathcal D_{X}(\phi_t)$ uniformly bounded from above for $t>\epsilon_0$.
\end{lem}

With the assumption of properness, Theorem \ref{0301} will follows from the above and the next lemmas.
\begin{lem}\label{0308}
For any solution $\phi_t$ of (\ref{0302}) with $t<1$,
$$\min_{\sigma\in H^c}\{I_{X}((\phi_t)_\sigma)-J_{X}((\phi_t)_\sigma)\}=I_{X}(\phi_t)-J_{X}(\phi_t).$$
\end{lem}
\begin{proof}
We will use the argument of Tian \cite{Tian297} to prove this lemma. For any $Y\in\mathfrak h^c$, let $\sigma(s)$ be the one parameter group generated by Re$(Y)$ with $\sigma(0)=id$. For solution $\phi_t$ of (\ref{0302}), set $\phi_{t\,s}=(\phi_t)_{\sigma(s)}$.
Note that (\ref{0302}) is equivalent to
\begin{eqnarray*}
h_{t}+(1-t)\phi_t=\log(1-\theta_X(\omega_{\phi_t}))+c_t,
\end{eqnarray*}
where $h_t$ is the normalized Ricci potential of $\omega_{\phi_t}$ and $c_t$ is a constant depending on $t$. Thus
\begin{eqnarray}\label{0309}
&&\left.{\frac{\partial}{\partial s}}\right|_{s=0}(I_{X}-J_{X})(\phi_{t\,s})\notag\\
&=&\int_M\left.{\frac{\partial}{\partial s}}\right|_{s=0}\phi_{t~s,\bar k}\phi_{t,i}g^{i\bar k}(1-\theta_X(\omega_{\phi_t}))\omega_{\phi_t}^n\notag\\
&=&-{\frac{1}{1-t}}\int_MY^i\left[h_{t,i}+{\frac{\theta_X(\omega_{\phi_t})_{,i}}{1-\theta_X(\omega_{\phi_t})}}\right](1-\theta_X(\omega_{\phi_t}))\omega_{\phi_t}^n\notag\\
&=&-{\frac{1}{1-t}}\int_MY(h_t)\omega_{\phi_t}^n+{\frac{1}{1-t}}\left(\int_MY(h_t)\theta_X(\omega_{\phi_t})\omega_{\phi_t}^n-\int_MY^i\theta_X(\omega_{\phi_t})_{,i}\omega_{\phi_t}^n\right).
\end{eqnarray}
Recall that $\theta_Y(\omega_{\phi_t})$ satisfies
\begin{eqnarray*}
\bigtriangleup_{\omega_{\phi_t}}\theta_Y(\omega_{\phi_t})+Y(h_{t})+\theta_Y(\omega_{\phi_t})=const.,
\end{eqnarray*}
thus
\begin{eqnarray*}
\int_MY(h_t)\theta_X(\omega_{\phi_t})\omega_{\phi_t}^n=-\int_M\theta_X(\omega_{\phi_t})\theta_Y(\omega_{\phi_t})\omega_{\phi_t}^n-\int_M\theta_X(\omega_{\phi_t})\bigtriangleup_{\omega_{\phi_t}}\theta_Y(\omega_{\phi_t})\omega_{\phi_t}^n.
\end{eqnarray*}
Substituting this into (\ref{0309}) and by integration by parts, it holds
\begin{equation}\label{0309+}
\left.{\frac{\partial}{\partial s}}\right|_{s=0}(I_{X}-J_{X})(\phi_{t\,s})=-{\frac{1}{1-t}}\int_MY(h_t)\omega_{\phi_t}^n-{\frac{1}{1-t}}\int_M\theta_X(\omega_{\phi_t})\theta_Y(\omega_{\phi_t})\omega_{\phi_t}^n=0.
\end{equation}
The last equality follows from (\ref{0201}). This shows that $s=0$ is a critical point of $(I_{X}-J_{X})(\phi_{t\,s})$.

To prove the lemma, it suffices to show that $(I_{X}-J_{X})(\phi_{t\,s})$ is convex with respect to $s$. It is direct to check that
\begin{eqnarray}\label{0310}
{\frac{\partial^2}{\partial^2 s}}\phi_{t\,s}=\left|\bar\partial\left({\frac{\partial}{\partial s}}\phi_{t\,s}\right)\right|_{\omega_{\phi_{t\,s}}}^2,
\end{eqnarray}
thus $\phi_{t\,s}$ gives a geodesic in the space of K\"ahler potentials.
In the following, we denote $\phi_s=\phi_{t\,s}$ for fixed $t$ for simplicity and $\omega_{\phi_{t\,s}}=\sqrt{-1}g_{i\bar j}(s)dz^i\wedge d\bar z^j$. Then
\begin{eqnarray}\label{0311}
{\frac{\partial}{\partial s}}\bigtriangleup_{\omega_{\phi_s}}\phi_{s}=-g^{i\bar k}g^{\bar jl}\cdot\phi_{s,\bar kl}\phi_{s,\bar ji}+\bigtriangleup_{\omega_{\phi_s}}\dot\phi_{s}.
\end{eqnarray}
Note that
\begin{eqnarray}\label{0312}
{\frac{d}{d s}}(I_{X}-J_{X})(\phi_{s})=-\int_{M}\dot\phi_s\bigtriangleup_{\omega_{\phi_s}}\phi_s(1-\theta_X(\omega_{\phi_s}))\omega_{\phi_s}^n+\int_M\dot\phi_sX(\phi_s)\omega_{\phi_s}^n.
\end{eqnarray}
We want to differentiate the above equality. For the first term, we have by (\ref{0311})
\begin{eqnarray*}
&&{\frac{d}{d s}}\int_{M}\dot\phi_s\bigtriangleup_{\omega_{\phi_s}}\phi_s(1-\theta_X(\omega_{\phi_s}))\omega_{\phi_s}^n\notag\\
&=&\int_{M}\ddot{\phi}_s\bigtriangleup_{\omega_{\phi_s}}\phi_s(1-\theta_X(\omega_{\phi_s}))\omega_{\phi_s}^n-\int_MX^i\dot\phi_{s,i}\bigtriangleup_{\omega_{\phi_s}}\phi_s\dot\phi_s\omega_{\phi_s}^n\notag\\
&&+\int_{M}\dot\phi_s\bigtriangleup_{\omega_{\phi_s}}\dot\phi_s\bigtriangleup_{\omega_{\phi_s}}\phi_s(1-\theta_X(\omega_{\phi_s}))\omega_{\phi_s}^n+\int_{M}\dot\phi_s\bigtriangleup_{\omega_{\phi_s}}\dot\phi_s(1-\theta_X(\omega_{\phi_s}))\omega_{\phi_s}^n\notag\\
&&-\int_{M}\dot\phi_s(\dot\phi_{s,l\bar k}\phi_{s,i\bar j})g^{i\bar k}g^{l\bar j}(1-\theta_X(\omega_{\phi_s}))\omega_{\phi_s}^n.
\end{eqnarray*}
Substituting (\ref{0310}) into the first term and by integration by parts, it follows
\begin{eqnarray}\label{0313}
&&{\frac{d}{d s}}\int_{M}\dot\phi_s\bigtriangleup_{\omega_{\phi_s}}\phi_s(1-\theta_X(\omega_{\phi_s}))\omega_{\phi_s}^n\notag\\
&=&-\int_M\dot\phi_s\dot\phi_{s,i}(\bigtriangleup_{\omega_{\phi_s}}\phi_s)_{,\bar k}g^{i\bar k}(1-\theta_X(\omega_{\phi_s}))\omega_{\phi_s}^n-\int_M\dot\phi_s\dot\phi_{s,l\bar k}\phi_{s,i\bar j}g^{i\bar k}g^{l\bar j}(1-\theta_X(\omega_{\phi_s}))\omega_{\phi_s}^n\notag\\
&&+\int_{M}\dot\phi_s\bigtriangleup_{\omega_{\phi_s}}\dot\phi_s(1-\theta_X(\omega_{\phi_s}))\omega_{\phi_s}^n\notag\\
&=&\int_M\dot\phi_{s,l}\dot\phi_{s,\bar k}\phi_{s,i\bar j}g^{l\bar j}g^{i\bar k}(1-\theta_X(\omega_{\phi_s}))\omega_{\phi_s}^n-\int_M\dot\phi_sX^i\dot\phi_{s,l}\phi_{s,i\bar j}g^{l\bar j}\omega_{\phi_s}^n\notag\\
&&+\int_{M}\dot\phi_s\bigtriangleup_{\omega_{\phi_s}}\dot\phi_s(1-\theta_X(\omega_{\phi_s}))\omega_{\phi_s}^n\notag\\
&=&\int_M\dot\phi_{s,l}\dot\phi_{s,\bar k}(\phi_{s,i\bar j}-g_{i\bar j})g^{i\bar k}g^{l\bar j}(1-\theta_X(\omega_{\phi_s}))\omega_{\phi_s}^n-\int_M\dot\phi_sX^i\dot\phi_{s,l}\phi_{s,i\bar j}g
^{l\bar j}\omega_{\phi_s}^n\notag\\
&&+\int_{M}X^i\dot\phi_{s,i}\dot\phi_s\omega_{\phi_s}^n,
\end{eqnarray}
where $g_{i\bar j}=g_{i\bar j}(s)$. The second term in (\ref{0312}) gives
\begin{eqnarray*}
{\frac{d}{d s}}\int_{M}\dot\phi_sX^i\phi_{s,i}\omega_{\phi_s}^n
=\int_M\ddot\phi_sX^i\phi_{s,i}\omega_{\phi_s}^n+\int_M\dot\phi_sX^i\dot\phi_{s,i}\omega_{\phi_s}^n+\int_M\dot\phi_sX^i\phi_{s,i}\bigtriangleup_{\omega_{\phi_s}}\dot\phi_s\omega_{\phi_s}^n.
\end{eqnarray*}
Substituting (\ref{0310}) into the above equality and by integration by parts again, we have
\begin{eqnarray}\label{0314}
{\frac{d}{d s}}\int_{M}\dot\phi_sX^i\phi_{s,i}\omega_{\phi_s}^n
=-\int_M\dot\phi_s\dot\phi_{s,l}(X^i\phi_{s,i})_{,\bar j}g^{l\bar j}\omega_{\phi_s}^n+\int_M\dot\phi_sX^i\dot\phi_{s,i}\omega_{\phi_s}^n.
\end{eqnarray}
Combining (\ref{0312})-(\ref{0314}), we get
\begin{eqnarray*}\label{0315}
{\frac{d^2}{d s^2}}(I_{X}-J_{X})(\phi_{s})=\int_{M}\dot\phi_{s,\bar k}\dot\phi_{s,l}g_{i\bar j}(0)g^{i\bar k}g^{l\bar j}(1-\theta_X(\omega_{\phi_s}))\omega_{\phi_s}^n\geq0.
\end{eqnarray*}
Hence, the lemma is proved.
\end{proof}

\begin{proof}[Proof of Theorem \ref{inverse proper}]
The theorem can be proved by using the properness principle of \cite{DR}. Suppose $\omega_0$ is the Mabuchi metric and ${Iso}_0(M,\omega_0)$ is the identity component of the corresponding isometry group. By a Calabi-Matsushima typed theorem of Mabuchi \cite{Mab4}, we have
\begin{eqnarray}\label{matsushima}
\mathfrak{aut}^X(M)=\mathfrak{iso}(M,\omega_0)\oplus J\mathfrak{iso}(M,\omega_0),
\end{eqnarray}
where $\mathfrak{aut}^X(M)$ and $\mathfrak{iso}(M,\omega_0)$ are Lie algebras of
$Aut_0^X(M)$ and ${Iso}_0(M,\omega_0)$, respectively.
We will check that $\mathcal D_X(\cdot)$, $Aut^X(M)$ satisfy (P1)-(P2), (P4)-(P7) in the Hypothesis 3.2 of \cite{DR}, which are enough for the "existence$\Rightarrow$ properness" direction:
\begin{itemize}
\item[(P1)] This is confirmed by \cite[Theorem 1.1]{Bern} and (2) of Lemma \ref{linear properties};
\item[(P2)] This can be shown by using (1) of Lemma \ref{linear properties} and Lemmas 5.15, 5.20, 5.29 of \cite{DR};
\item[(P4)] This is \cite[Lemma 5.9]{DR};
\item[(P5)] This is shown in \cite[Appendix 4]{Mab2};
\item[(P6)] This can be shown exactly as in \cite[Theorem 8.1]{DR}, by using (\ref{matsushima}) instead of \cite[Proposition 6.10]{DR};
\item[(P7)] This follows from the co-cycle condition of $\mathcal D_X(\cdot)$.
\end{itemize}
The theorem then follows from the second part of \cite[Theorem 3.4]{DR}.
\end{proof}

\section{Existence criterion on Fano group compactifications}
In this section, we will prove Theorem \ref{bar condition thm}.
Let $M$ be a group compactification and $\omega_0$ be a $K\times K$-invariant K\"ahler metric in $2\pi c_1(M)$. Assume $\omega_0=\sqrt{-1}\partial\bar\partial\psi_0$ on $G$.
For  $\phi\in\mathcal H_{K\times K}(\omega_0)$, we will write $\psi_\phi$ in short for $\psi_0+\phi$ and $u_\phi$ the Legendre function of $\psi_\phi$.

\subsection{Reduction of the modified Ding functional}
We will give a formula of $\mathcal D_{X}(\phi)$ in terms of $\phi$ and $u_\phi$.
First, we compute the Futaki invariant of a vector field in $\mathfrak z(\mathfrak g)$.
\begin{lem}\label{Futaki comput}Let $Y$ be a vector field of form
\begin{eqnarray}\label{coefficients Y}
Y=\sqrt{-1}Y^iE_0^i,~1\leq i\leq r
\end{eqnarray}
for some $Y^i\in\mathbb C$ such that $\alpha_iY^i=0$ for any $\alpha\in\Phi$. Then
\begin{eqnarray}\label{4103}
Fut(Y)=-V\cdot Y^i\mathbf b_i,
\end{eqnarray}
where $\mathbf b={\frac1V}\int_{2P_+}y\pi(y) \,dy$
is the barycentre of $2P_+$ with respect to the measure $\pi(y)\,dy$.
\end{lem}
\begin{proof}
Since $Y\in\mathfrak z(\mathfrak g)$, it is $K\times K$-invariant, so is its potential. Recall that
\begin{eqnarray}\label{4102}
Fut(Y)=-\int_M\hat\theta_Y(\omega_0)\omega_0^n,
\end{eqnarray}
where $\hat\theta_Y(\omega_0)$ is the potential of $Y$ normalized by
\begin{eqnarray}\label{4104}
\int_M\hat\theta_Y(\omega_0)e^{h_0}\omega_0^n=0.
\end{eqnarray}
By  $\omega_0=\sqrt{-1}\partial\bar\partial\psi_0$ and Lemma \ref{Derivative}, it is not hard to see that
\begin{eqnarray}\label{4101}
\hat\theta_Y(\omega_0)=Y^i{\frac{\partial}{\partial x^i}}\psi_0+C,~\forall x\in\mathfrak a_+,
\end{eqnarray}
where $C$ is a constant determined by (\ref{4104}). On the other hand, we have
\begin{eqnarray}\label{hat nor potential}
\int_{\mathfrak a_{+}}Y^i{\frac{\partial}{\partial x^i}}\psi_0\,e^{h_0}\det(\psi_{0,ij})\prod_{\alpha\in \Phi_+}\langle\alpha,\nabla \psi_0\rangle^2\,dx
&=&\int_{\mathfrak a_{+}}Y^i{\frac{\partial}{\partial x^i}}\psi_0\,e^{-\psi_0}\mathbf J(x)\,dx\notag\\
&=&-\int_{\mathfrak a_{+}}Y^i{\frac{\partial}{\partial x^i}}\left(e^{-\psi_0}\mathbf J(x)\right)\,dx,
\end{eqnarray}
here we used the fact that
\begin{eqnarray*}
Y^i{\frac{\partial}{\partial x^i}}\mathbf J(x)=2\mathbf J(x)\sum_{\alpha\in \Phi_+}Y^i\alpha_i\cdot\coth\langle\alpha,x\rangle\equiv0.
\end{eqnarray*}
Note that when $M$ is Fano, $4\rho\in\text{Int}(2P_+)$ (cf. \cite[Remark 4.10]{Del1} or \cite[\S 3.2]{LZZ}), by \cite[Proposition 2.10]{Del2}. Hence, we have
\begin{eqnarray*}
e^{-\psi_0}\mathbf J(x)=e^{4\rho(x)-\psi_0}\prod_{\alpha\in\Phi_+}\left({\frac{1-e^{-2\alpha(x)}}{2}}\right)^2\to 0,~x\to \infty \text{ in }\mathfrak a_+.
\end{eqnarray*}
Also recall the fact that $\mathbf J(x)=0$ on $\partial(\mathfrak a_{+})$. By integration by parts in (\ref{hat nor potential}), we see that
\begin{eqnarray*}
\int_{\mathfrak a_{+}}Y^i{\frac{\partial}{\partial x^i}}\psi_0e^{h_0}\det(\psi_{0,ij})\prod_{\alpha\in \Phi_+}\langle\alpha,\nabla \psi_0\rangle^2\,dx=0.
\end{eqnarray*}
Thus by Proposition \ref{KAK int}, we get $C=0$ in (\ref{4101}). (\ref{4103}) then follows from (\ref{4102}).
\end{proof}

Then we use (\ref{0201}) to determine the potential of the extremal vector field $X$.
\begin{lem}
Under the coordinates chosen in \S\ref{grp cpt}, the extremal field $X$, when restricted on $Z$, can be expressed by
\begin{eqnarray}\label{coefficients X}
X=\sqrt{-1}X^iE_0^i,~1\leq i\leq r
\end{eqnarray}
for some $X^i\in\mathbb R$ such that $\alpha(X)=0,~\forall \alpha\in\Phi$. Furthermore, $X^i$'s are determined by the condition
\begin{eqnarray}\label{0405}
\int_{2P_+}v^iy_i(1-\theta_X(y))\pi(y)\,dy=0,~\forall v\in\mathfrak z(\mathfrak g),
\end{eqnarray}
where $\theta_X(y)=X^iy_i-X^i\mathbf b_i$.
\end{lem}
\begin{proof}
Since Futaki invariant is a character on $\eta_r(M)$, it suffices to consider (\ref{0201}) for all $v\in\mathfrak z(\eta_r(M))\subset\mathfrak z(\mathfrak g)$. We may assume $X$ is of form (\ref{coefficients X}). Since $K_X$ lies in a compact group, we have $X^i\in\mathbb R$.

For $\phi\in\mathcal H_{K\times K}(\omega_0)$ and $v\in\eta_c(M)$, $\theta_v(\omega_\phi)$ is $K\times K$-invariant, so it can be written as
\begin{equation}\label{+5201}
\theta_{v}(\omega_{\phi})=v^i{\frac{\partial\psi_\phi}{\partial x^i}}+c_v,\nonumber
\end{equation}
where $v^i$ and $c_v$ are constants with $v^i\alpha_i=0$ for any $\alpha\in\Phi_+$.

By the second equality of (\ref{0201+}), the potential is determined by
\begin{eqnarray}\label{potential v}
\theta_v(y)=v^iy_i-v^i\mathbf b_i.
\end{eqnarray}
Let $r_z=\dim(\mathfrak z(\mathfrak g))$ and suppose $E_1^0,...,E_{r_z}^0$ be a basis of $\mathfrak z(\mathfrak g)$. We claim that the extremal vector field $X$ is given by $X=\displaystyle\sum_1^{r_z}\sqrt{-1}X^iE_i^0\in\mathfrak z(\mathfrak g)$ such that
\begin{eqnarray}\label{equation X}
\mathbf b_i={\frac1V}\left(\int_{2P_+}y_iy_j\pi(y)\,dy-V\mathbf b_i\mathbf b_j\right)X^j,\ 1\leq i,j\leq r_z.
\end{eqnarray}
In view of (\ref{potential v}) and Lemma \ref{Futaki comput}, it is direct to check that $X$ given by (\ref{equation X}) satisfies (\ref{0201}). Hence $X$ must be extremal by the uniqueness. To see that (\ref{equation X}) has a unique solution, it suffices to check that the matrix $(a_{ij})$ given by
$$a_{ij}={\frac1V}\int_{2P_+}y_iy_j\pi(y)\,dy-\mathbf b_i\mathbf b_j$$
is invertible. In fact, for any vector $v=(v^i)$, consider the convex function $f_v(y)=(v^iy_i)^2$. By Jensen inequality,
$$v^iv^ja_{ij}={\frac1V}\int_{2P_+}[v(y)]^2\pi(y)\,dy-[v(\mathbf b)]^2\geq 0,$$
with equality if and only if $f_v(y)$ is affine on $2P_+$.
However, this forces $v=0$, thus $(a_{ij})>0$. This completes the proof.
\end{proof}

\begin{prop}\label{Ding energy}
For $\phi\in\mathcal H_{K\times K}(\omega_0)$, the modified Ding functional is given by
$$\mathcal D_{X}(\phi)=\mathcal L_{X}(u_\phi)+\mathcal F(u_\phi)+const.,$$
where
\begin{eqnarray}
\mathcal L_{X}(u_\phi)&=&{\frac1V}\int_{2P_+}u_\phi(y)\pi(y)[1-\theta_X(y)]\,dy-u_\phi(4\rho),\label{L(u)}\\
\mathcal F(u_\phi)&=&-\log\left(\int_{\mathfrak a_+}e^{-\psi_\phi}\mathbf J(x)\,dx\right)+u_\phi(4\rho).\label{F(u)}
\end{eqnarray}
\end{prop}

\begin{proof}
By (\ref{0203}), Proposition \ref{KAK int} and (\ref{MA}), it follows\footnote{Since we have assumed $C_H=1$, it follows $V:=\int_M\omega_0^n=\int_{2P_+}\pi(y)\,dy$ by Proposition \ref{KAK int}. Similarly $\int_{2P_+}[1-\theta_X(y)]\pi(y)\,dy=V$.}
\begin{eqnarray}
\mathcal D^0_{X}(\phi)&=&-{\frac1V}\int_0^1\int_{\mathfrak a_+}\dot\phi_s[1-\theta_X(\omega_{\phi_s})]\det(\psi_{\phi_s,ij})\prod_{\alpha\in\Phi_+}\langle\alpha,\nabla\psi_{\phi_s}\rangle^2dx\wedge ds\notag\\
&&+const.,\label{0409}\\
\mathcal N(\phi)&=&-\log\left({\frac1V}\int_{\mathfrak a_+}e^{-\psi_{\phi}}\mathbf J(x)\,dx\right).\notag
\end{eqnarray}
By differetiation with Legendre transformations, we have
$\dot u_s(y_s(x))=-\dot\psi_s(x)$.
Then by \eqref{0409}, $\mathcal D^0(\phi)$ equals
\begin{eqnarray*}
\int_0^1\int_{2P_+}\dot u_s[1-\theta_X(y)]\pi(y)\,dy\wedge ds=\int_{2P_+} u_\phi[1-\theta_X(y)]\pi(y)\,dy\wedge ds+const.
\end{eqnarray*}
The proposition is proved.
\end{proof}

\subsection{The linear part}\label{linear}
In this part, we deal with the linear part $\mathcal L _X(\cdot)$. First, we introduce the spaces of normalized functions.
Let $O$ be the origin of $\mathfrak a^*$. Note that $\mathfrak a^*_t$ is the fixed point set of the $W$-action. Thus $\nabla u(O)\in\mathfrak a^*_t$ for any $u\in\mathcal C_{W}$. We normalize $u\in\mathcal C_{W}$ by
\begin{eqnarray*}
\hat u(y)=u(y)-\langle\nabla u(O),y\rangle-u(O).
\end{eqnarray*}
Clearly $\hat u\in\mathcal C_{W}$ and
\begin{align}\label{normalization-u}\min_{2P}\hat u=\hat u(O)=0.
\end{align}
The subset of normalized functions in $\mathcal C_{W}$  will be denoted by $\hat {\mathcal C}_{W}$.

\begin{prop}\label{linear proper thm}
Under the assumption $c_X>0$ and (\ref{bar condition}), there exists a constant $\lambda>0$ such that
\begin{eqnarray}\label{linear prop}
\mathcal L_{X}(u)\geq \lambda\int_{2P_+}u\pi(y)[1-\theta_X(y)]\,dy,~\forall u\in\hat{\mathcal C}_W.
\end{eqnarray}
\end{prop}

\begin{proof}
Suppose the proposition is not true, then there's a sequence $\{u_k\}\subset\hat{\mathcal C_W}$
such that
\begin{eqnarray}\label{04 assumption}
\begin{cases}
\mathcal L_{X}(u_k)\to 0,&\\
\int_{2P_+}u_k\pi(y)[1-\theta_X(y)]\,dy=1.&
\end{cases}
\end{eqnarray}
By $c_X>0$ and the argument of \cite[Lemma 6.1]{LZZ}, the second equality  implies  there is a subsequence(still denoted by $\{u_k\}$) which converges locally uniformly to some $u_{\infty}\in\hat{\mathcal C}_W$.

For any $u\in\mathcal C$, by convexity, we have
\begin{eqnarray}\label{0402}
u-\langle\nabla u(\mathbf b_X),y-\mathbf b_X\rangle-u(\mathbf b_X)\geq0,
\end{eqnarray}
thus
\begin{eqnarray}\label{0403}
\mathcal L_{X}(u)&=&{\frac1V}\int_{2P_+}[u-\langle\nabla u(\mathbf b_X),y-\mathbf b_X\rangle-u(\mathbf b_X)]\pi(y)[1-\theta_X(y)]\,dy\notag\\
&&+{\frac1V}\int_{2P_+}[\langle\nabla u(\mathbf b_X),y-\mathbf b_X\rangle+u(\mathbf b_X)]\pi(y)[1-\theta_X(y)]\,dy-u(4\rho)\notag\\
&=&{\frac1V}\int_{2P_+}[u-\langle\nabla u(\mathbf b_X),y-\mathbf b_X\rangle-u(\mathbf b_X)]\pi(y)[1-\theta_X(y)]\,dy+u(\mathbf b_X)- u(4\rho)\notag\\
&\geq&{\frac1V}\int_{2P_+}[u-\langle\nabla u(\mathbf b_X),y-\mathbf b_X\rangle-u(\mathbf b_X)]\pi(y)[1-\theta_X(y)]\,dy\notag\\
&&+\langle\nabla u(4\rho),\mathbf b_X-4\rho\rangle\geq0,
\end{eqnarray}
where the last inequality follows from (\ref{bar condition}), (\ref{0402}) and the fact that $\nabla u(4\rho)\in\mathfrak a_+$. Applying the above inequality to $u_k$, by (\ref{04 assumption}), we have
\begin{eqnarray}\label{0404}
0&\leq& \int_{2P_+}[u_k-\langle\nabla u_k(\mathbf b_X),y-\mathbf b_X\rangle-u_k(\mathbf b_X)]\pi(y) [1-\theta_X(y)]\,dy\to 0, \label{0404} \\
0&\leq&\langle\nabla u_k(4\rho),\mathbf b_X-4\rho\rangle\to 0. \label{0404s}
\end{eqnarray}
By \eqref{0404}, we see that $u_\infty$ must be affine linear. Since $u_k(O)=0$, we have $u_\infty(y)=\xi^iy_i$ for some $(\xi^i)\in\bar{\mathfrak a}_+$. Since $u_\infty$ is normalized and $O$ lies  in the interior of $2P_+\cap\mathfrak a_+^*$, it holds $\xi\in\mathfrak a_{ss}$. Otherwise $u_\infty$ is not nonnegative. Substituting $u_\infty$ into \eqref{0404s}, we see that
$\langle\xi,\mathbf b_X-4\rho\rangle=0$.
But $\xi\in\bar{\mathfrak a}_+$ and $\mathbf b_X-4\rho\in\Xi$. Hence $\xi^i=0$ and consequently $u_\infty(y)\equiv0$.

Since $u_k(4\rho)\to u_\infty(4\rho)=0$, by \eqref{L(u)}
and the second line of (\ref{04 assumption}), we have
$\mathcal L_{X}(u_k)\to 1$
by the second line of (\ref{04 assumption}), which is a contradiction. Thus the proposition is proved.
\end{proof}

Yao use (\ref{linear prop}) to define the "uniform relative Ding stablity"  in toric case \cite{Yao}. In \cite{Yao}, it is shown the condition $c_X>0$ is a necessary condition of (\ref{linear prop}). Since the arguments of \cite{Yao} can be generalized to group compactifications with no difficulties, we omit the details.
\begin{prop}
Inequality (\ref{linear prop}) can not hold if $c_X\leq0$.
\end{prop}


\subsection{Sufficiency}\label{sufficiency}
We first show the sufficient part of Theorem \ref{bar condition thm} by using Theorem \ref{0301}.
It suffices to prove the following theorem.
\begin{thm}\label{proper}
If $c_X>0$ and (\ref{bar condition}) holds, then the modified Ding functional is proper modulo $Z(G)$. Consequently, $M$ admits Mabuchi metrics by Theorem \ref{0301}.
\end{thm}

First we have the following lemma on non-linear part.
\begin{lem}\label{non-linear lem}
For any $\phi\in\mathcal H_{K\times K}(\omega_0)$, let $$\tilde\psi_\phi:=\psi_\phi-4\rho_ix^i,~x\in\mathfrak a_+.$$
Then
\begin{eqnarray}\label{0406}
\mathcal F(u_\phi)=-\log\left(\int_{\mathfrak a_+}e^{-(\tilde\psi_\phi-\inf_{\mathfrak a_+}\tilde\psi_\phi)}\prod_{\alpha\in\Phi_+}\left(\frac{1-e^{-2\alpha_ix^i}}{2}\right)^2dx\right).
\end{eqnarray}
Consequently, for any $c>0$,
\begin{eqnarray}\label{0407}
\mathcal F(u_\phi)\geq \mathcal F\left(\frac{u_\phi}{1+c}\right)+n\cdot\log(1+c).
\end{eqnarray}
\end{lem}
\begin{proof}
Since $\psi_\phi$ is convex, so is $\tilde\psi_\phi$. Thus if $x^*\in\mathfrak a_+$ satisfies $\nabla\psi_\phi(x^*)=4\rho$, then
$$\tilde\psi_\phi(x)\geq\tilde\psi_\phi(x^*)=\inf_{\mathfrak a_+}\tilde\psi_\phi.$$
By the definition of Legendre transformation, we have
\begin{eqnarray*}
\psi_\phi(x)+u_\phi(4\rho)=\psi_\phi(x)+4x^{*i}\rho_i-\psi_\phi(x^*)
=\psi_\phi(x)-\inf_{\mathfrak a_+}\tilde\psi_\phi.
\end{eqnarray*}
Substituting this into (\ref{F(u)}), it follows
\begin{eqnarray*}
\mathcal F(u_\phi)&=&-\log\left(\int_{\mathfrak a_+}e^{-(\psi_\phi+u_\phi(4\rho))}\mathbf J(x)dx\right)\notag\\
&=&-\log\left(\int_{\mathfrak a_+}e^{-(\tilde\psi_\phi-\inf_{\mathfrak a_+}\tilde\psi_\phi)}e^{-4\rho_ix^i}\mathbf J(x)dx\right)\notag\\
&=&-\log\left(\int_{\mathfrak a_+}e^{-(\tilde\psi_\phi-\inf_{\mathfrak a_+}\tilde\psi_\phi)}\prod_{\alpha\in\Phi_+}\left(\frac{1-e^{-2\alpha_ix^i}}{2}\right)^2dx\right).
\end{eqnarray*}
This proves (\ref{0406}).

Then we prove (\ref{0407}). For $u_c(y)={\frac1{1+c}}u(y)$, its Legendre function $\psi_c(x)={\frac1{1+c}}\psi((1+c)x)$ satisfies
$\tilde{\psi}_c(x)={\frac1{1+c}}\tilde\psi((1+c)x)$.
In particular,
$$-\inf_{\mathfrak a_+}\tilde{\psi}_c(x)=-{\frac1{1+c}}\inf_{\mathfrak a_+}\tilde\psi.$$
By the above relations and (\ref{0406}), one gets
\begin{eqnarray*}\label{0408}
\mathcal F(\hat u)&=&-\log\left(\int_{\mathfrak a_+}e^{-{\frac{1}{1+c}}(\tilde\psi_\phi((1+c)x)-\inf_{\mathfrak a_+}\tilde\psi_\phi)}\prod_{\alpha\in\Phi_+}\left(\frac{1-e^{-2\alpha_ix^i}}{2}\right)^2dx\right)\notag\\
&=&-\log\left(\int_{\mathfrak a_+}e^{-{\frac{1}{1+c}}(\tilde\psi_\phi(x)-\inf_{\mathfrak a_+}\tilde\psi_\phi)}\prod_{\alpha\in\Phi_+}\left(\frac{1-e^{-\frac2{1+c}\alpha_ix^i}}{2}\right)^2\, dx\right)+r\cdot\log(1+c).\notag\\
\end{eqnarray*}
Note that $\#\Phi=n-r$. Combining the above inequality and relations
\begin{eqnarray*}
\log(1+c)\geq\log(1-e^{-t})-\log(1-e^{\frac t{1+c}})\geq0,~\forall t,c\geq0
\end{eqnarray*}
and
\begin{eqnarray*}
\mathcal F(u)\geq-\log\left(\int_{\mathfrak a_+}e^{-{\frac{1}{1+c}}(\tilde\psi_\phi(x)-\inf_{\mathfrak a_+}\tilde\psi_\phi)}\prod_{\alpha\in\Phi_+}\left(\frac{1-e^{-\frac2{1+c}\alpha_ix^i}}{2}\right)^2\,dx\right).
\end{eqnarray*}
 Hence, we have (\ref{0407}).
\end{proof}

\begin{prop}\label{proper mod thm}
Suppose $c_X>0$ and (\ref{bar condition}) holds. Then there are constants $c,C>0$ such that
\begin{eqnarray}\label{proper mod}
\mathcal D_{X}(u)\geq c\int_{2P_+}u[1-\theta_X(y)]\pi(y)\,dy-C,~\forall u\in\hat{\mathcal C}_W.
\end{eqnarray}
\end{prop}
\begin{proof}
Define a function $A$ by
\begin{eqnarray*}
A(y)={\frac{V}{\int_{\mathfrak a_+}e^{-\psi_0}\mathbf J(x)dx}}e^{h_0(\nabla u_0(y))},~y(x)=\nabla\psi_0(x).
\end{eqnarray*}
It is clear that
$$\int_{\mathfrak a_+}e^{-\psi_0}\mathbf J(x)dx=\int_Me^{h_0}\omega_0^n=V.$$
Hence, $A$ is a bounded smooth function.

Let
\begin{eqnarray*}
\mathcal D_A(u_\phi):=\mathcal D^0_A(u_\phi)+\mathcal N(\phi),~\forall\phi\in\mathcal H_{K\times K}(\omega_0),
\end{eqnarray*}
where $$\mathcal D^0_A(u):={\frac1V}\int_{2P_+}uA(y)\pi(y)\,dy.$$
It is obvious that $u_0$ is a critical point of $\mathcal D_A(\cdot)$. On the other hand, along any geodesic, $\mathcal D^0_A(\cdot)$ is affine by Lemma \ref{0601} and $\mathcal N(\cdot)$ is convex by \cite[Theorem 1.1]{Bern}. Hence,
\begin{eqnarray}\label{D_A lower bound}
\mathcal D_A(u)\geq\mathcal D_A(u_0),~\forall u\in\hat{\mathcal C}_W.
\end{eqnarray}

Rewrite $\mathcal D_A(\cdot)=\mathcal L_A(\cdot)+\mathcal F(\cdot)$, where
$$\mathcal L_A(u):={\frac1V}\int_{2P_+}uA(y)\pi(y)\,dy-u(4\rho).$$
By Proposition \ref{linear proper thm} and the boundedness of $A$, it is clear that for any $\delta>0$
\begin{eqnarray*}
\left|\mathcal L_{X}(u)-\mathcal L_A(u)\right|&=&\left|\int_{2P_+}u(1-\theta_X(y)-A(y))\pi(y)\,dy\right|\notag\\
&\leq&C_A\int_{2P_+}u[1-\theta_X(y)]\pi(y)\,dy\notag\\
&\leq&{\frac{C_A(1+\delta)}{\lambda}}\mathcal L_{X}(u)-C_A\delta\int_{2P_+}u[1-\theta_X(y)]\pi(y)\,dy,~\forall u\in\hat{\mathcal C}_W,
\end{eqnarray*}
for some constant $C_A>0$. Then
\begin{eqnarray*}
\mathcal L_{X}(u)\geq{\frac{\lambda}{\lambda+C_A(1+\delta)}}\left[\mathcal L_A(u)+C_A\delta\int_{2P_+}u[1-\theta_X(y)]\pi(y)\,dy\right],~\forall u\in\hat{\mathcal C}_W.
\end{eqnarray*}
Hence, taking $C=\frac{C_A(1+\delta)}{\lambda}$,  we have for any $u\in\hat{\mathcal C}_W$,
\begin{eqnarray*}\label{0408}
\mathcal D_X(u)
&\geq&\mathcal L_A\left({\frac{u}{1+C}}\right)+\mathcal F(u)+{\frac{ C_A\delta}{1+C_AC}}\int_{2P_+}u[1-\theta_X(y)]\pi(y)\,dy\notag\\
&\geq&\mathcal L_A\left({\frac{u}{1+C}}\right)+\mathcal F\left({\frac{u}{1+C}}\right)+{\frac{C_A\delta}{1+C_AC}}\int_{2P_+}u[1-\theta_X(y)]\pi(y)\,dy-n\log\left(1+C\right)\notag\\
&=&\mathcal D_A\left({\frac{u}{1+C}}\right)+{\frac{ C_A\delta}{1+C_AC}}\int_{2P_+}u[1-\theta_X(y)]\pi(y)\,dy-n\log\left(1+C_AC\right)\notag\\
&\geq&{\frac{ C_A\delta}{1+C_AC}}\int_{2P_+}u[1-\theta_X(y)]\pi(y)\,dy+\left(\mathcal D_A(u_0)-n\log\left(1+C_AC\right)\right),
\end{eqnarray*}
where we used (\ref{0406}) and (\ref{D_A lower bound}).
This completes the proof.
\end{proof}

To use Theorem \ref{0301}, we introduce the following normalization: for any $\phi\in \mathcal H_{K\times K}(\omega_0)$, let $u_\phi$ be the Legendre function of $\psi_\phi$.
Take a $v\in\eta_c(M)$ such that Re$(v)=-\nabla u_\phi(O)$. Let $\sigma_v(t)$ be the one parameter group generated by Re$(v)$. Then $\sigma_v(t)\in Z(G)$. It follows
$$(\sigma_v(1))^*\omega_\phi=\omega_0+ \sqrt{-1}\partial\bar\partial \hat\phi$$
induces a $K\times K$-invariant K\"ahler potential $\hat\phi$. Since we may also normalize $\psi_{\hat\phi}$ so that $\psi_{\hat{\phi}}(O)=0$, thus the Legendre function  $u_{\hat\phi}$  of $\psi_{\hat \phi}$ is given by
\begin{eqnarray}\label{0417+}
u_{\hat\phi}(y)=u_{\phi}(y)-\langle\nabla u_\phi(O),y\rangle-u_{\phi}(O),
\end{eqnarray}
which satisfies $u_{\hat\phi}\in\hat{\mathcal C}_W$.
Then we have
\begin{lem}\label{invariant under normalization}
Under the above normalization, we have
$\mathcal D_{X}(u_{\hat\phi})=\mathcal D_{X}(u_\phi)$.
\end{lem}
\begin{proof}
Denote $a^i=-\text{Re}(u_{\phi,i}(O))$, then $(a^i)\in\mathfrak a_t$ and consequently $\alpha(a)=0$ for all $\alpha\in\Phi$. On the other hand, we have
$$\psi_{\hat\phi}(x)=\psi_\phi(x-a)-u_\phi(O).$$
Taking change of variables $x\to(x-a)$ in (\ref{0406}), by the above relations, we see that
$\mathcal F(u_\phi)=\mathcal F(u_{\hat\phi})$.
By $(a^i)\in\mathfrak a_t$ and (\ref{0405}),
$\mathcal L_{X}(a^iy_i-u_\phi(O))=0$.
Hence, by (\ref{0417+}) $\mathcal L_{X}(u_\phi)=\mathcal L_{X}(u_{\hat\phi})$. The lemma is proved.
\end{proof}

The following lemma is analogous to \cite[Lemma 4.14]{LZZ} and \cite[Lemma 3.4]{WZZ}, we omit the proof.
\begin{lem}\label{609}
There exists a uniform $C_J>0$ such that
\begin{equation}\label{bound psi}
\left|J_{X}(\hat\phi)-\int_{2P_+}u_{\hat\phi}[1-\theta_X(y)]\pi(y)\,dy\right|\leq C_J, \ \forall \phi\in \mathcal H_{K\times K}(\omega_0),\nonumber
\end{equation}
where $u_{\hat\phi}\in\hat{\mathcal C}_{W}$ and $\psi_{\hat{\phi}}$ is the Legendre function of $u_{\hat\phi}$.
\end{lem}

\begin{proof}[Proof of Theorem \ref{proper}]
For any $\phi\in \mathcal H_{K\times K}(\omega_0)$, there exists $\sigma\in Z(G)$ such that
$$\sigma^*\omega_\phi=\omega_0+ \sqrt{-1}\partial\bar\partial \hat\phi$$
as above.
Applying Proposition \ref{proper mod thm}, we have
$$\mathcal D_{X}(u_{\hat\phi})\geq  \delta\int_{2P_+}u_{\hat\phi}\pi \,dy-C_{\delta}.$$
Thus  by  Proposition \ref{Ding energy}, Lemma \ref{invariant under normalization} and Lemma \ref{609},
$$\mathcal D_{X}(\phi)=\mathcal D_{X}(\hat\phi)=\mathcal D_{X}(u_{\hat\phi})  \geq \delta\cdot J_{X}(\hat\phi)-C_J-C_{\delta}.$$
The theorem then follows from (\ref{0204}).
\end{proof}

\subsection{Necessity}\label{necessity}
To complete the proof of Theorem \ref{bar condition thm}, we will show that (\ref{bar condition}) is also a necessary condition of the existence of Mabuchi metrics. It is equivalent to show that
\begin{eqnarray}\label{040X necessity}
\langle\xi,\mathbf b_X-4\rho\rangle>0,~\forall \xi\in\mathfrak a_+.
\end{eqnarray}
We will adopt the method used in \cite{Del2}.

By the $K\times K$-invariance, \eqref{0202} can be reduced to the following Monge-Amp\'ere equation on $\mathfrak a_+$,
\begin{eqnarray}\label{MA on a_+}
\det(\psi_{0,ij}+\phi_{ij})\prod_{\alpha\in\Phi_+}\langle\alpha,\nabla(\psi_0+\phi)\rangle^2=C\cdot\frac{ e^{-(\psi_0+\phi-\log\mathbf J)}}{1-\theta_X(\omega_0)-X(\phi)}.
\end{eqnarray}
Suppose $\phi$ is a solution, for any $\xi\in\mathfrak a_+$, we have
\begin{eqnarray}
0&=&-\int_{\mathfrak a_+}\xi^i{\frac{\partial}{\partial x^i}}e^{-(\psi_0+\phi-\log\mathbf J)}\notag\\
&=&\int_{\mathfrak a_+}\xi^ie^{-(\psi_0+\phi-\log\mathbf J)}{\frac{\partial (\psi_0+\phi-\log\mathbf J)}{\partial x^i}}\notag\\
&=&\int_{\mathfrak a_+}\xi^i\det(\psi_{0,ij}+\phi_{,ij})\prod_{\alpha\in\Phi_+}\langle\alpha,\nabla(\psi_0+\phi)\rangle^2[1-\theta_X(\omega_\phi)]{\frac{\partial(\psi_0+\phi-\log\mathbf J)}{\partial x^i}}\notag\\
&<&\int_{\mathfrak a_+}\xi^i\det(\psi_{0,ij}+\phi_{,ij})\prod_{\alpha\in\Phi_+}\langle\alpha,\nabla(\psi_0+\phi)\rangle^2[1-\theta_X(\omega_\phi)]{\frac{\partial (\psi_0+\phi-\log\mathbf J)}{\partial x^i}}\notag\\
&=&V\int_{2P_+}\xi^i(y_i-4\rho_i)\pi(y)[1-\theta_X(y)]\,dy\notag\\
&=&V\cdot\langle\xi,\mathbf b_X-4\rho\rangle,
\end{eqnarray}
where in the fourth line we used the fact that for any $\xi,x\in\mathfrak a_+$
\begin{eqnarray*}
-\xi^i{\frac{\partial}{\partial x^i}}\log\mathbf J=-2\sum_{\alpha\in\Phi_+}\alpha(\xi)\cdot\coth\alpha(x)
<-2\sum_{\alpha\in\Phi_+}\alpha(\xi)
=-4\rho(\xi).
\end{eqnarray*}
Then we have (\ref{040X necessity}).

\section{Appendix: Proof of Theorem \ref{0501}}

In this appendix, we solve (\ref{0302}) at $t=0$. Following \cite{Zhu} for the K\"ahler-Ricci soliton case,
we introduce the following path,
\begin{eqnarray}\label{0502}
(1-\theta_X(\omega_{\phi_t}))^t\omega_{\phi_t}^n=e^{h_0}\omega_0^n,~t\in[0,1].
\end{eqnarray}
Set $\mathfrak I:=\{t\in[0,1]|(\ref{0502}) \text{ has a solution for $t$}\}$.
The Calabi-Yau theorem implies that $0\in\mathfrak I$. We shall prove $\mathfrak I$ is both open and closed in $[0,1]$.

\subsection{Openness}
Define a functional
$$J_t(\phi)=\int_0^1\int_M\dot\phi_s(1-\theta_X(\omega_{\phi_s}))^t\omega_{\phi_s}^n,$$
where $\phi_s$ is any smooth path joining $\phi$ and $0$ in $\mathcal H_{X}(\omega_0)$.
It is standard to show that $J_t(\cdot)$ is well-defined. Thus by taking $\phi_s=s\phi$, we have
\begin{eqnarray*}
J_t(\phi)=\int_0^1\int_M\phi(1-\theta_X(\omega_{d\phi_s}))^t\omega^n_{s\phi}.
\end{eqnarray*}
Denote an operator by
$$L_{t}(\psi):=\bigtriangleup_{\omega_{\phi_t}}\psi-{\frac{tX(\psi)}{1-\theta_X(\omega_{\phi_t})}}-\int_M\psi(1-\theta_X(\omega_{\phi_t}))^t\omega_{\phi_t}^n, \ \forall\psi\in\mathcal H_{X}(\omega_0).$$
Then for any $K_X$-invariant smooth real functions $f$ and $g$, it is easy to see
\begin{eqnarray}\label{0503}
\int_ML_t(f)g(1-\theta_X(\omega_{\phi_t}))^t\omega_{\phi_t}^n=\int_MfL_t(g)(1-\theta_X(\omega_{\phi_t}))^t\omega_{\phi_t}^n.
\end{eqnarray}
We have
\begin{lem}\label{0504}
Suppose $\phi_t$ is a smooth solution of (\ref{0502}) for some $t\in[0,1)$, then the first eigenvalue of $L_{t}$ is positive.
\end{lem}
\begin{proof}
Suppose $\lambda$ is the first eigenvalue and $\psi$ is an eigenfunction. Then by
$L_t\psi=-\lambda\psi$,
\begin{eqnarray*}
\lambda\int_M\psi(1-\theta_X(\omega_{\phi_t}))^t\omega_{\phi_t}^n=\int_M\psi(1-\theta_X(\omega_{\phi_t}))^t\omega_{\phi_t}^n\cdot\int_M\int_M\psi(1-\theta_X(\omega_{\phi_t}))^t\omega_{\phi_t}^n.
\end{eqnarray*}
By the assumption $c_X>0$, if $\psi\equiv c$ for some constant $c\not=0$, then $\lambda>0$. Thus we may assume that $\psi\not\equiv const.$ below.

As before, we may choose a local co-frame $\{\Theta^i\}$ such that $\omega_{\phi_t}=\sqrt{-1}\sum_{i=1}^n\Theta^i\wedge\bar\Theta^i$. By (\ref{0305}) and integration by parts, it follows
\begin{eqnarray}\label{0506}
&&\lambda\int_M\psi_{,i}\psi_{,\bar i}(1-\theta_X(\omega_{\phi_t}))^t\omega_{\phi_t}^n\notag\\
&=&-\int L(\psi)_{,i}\psi_{,\bar i}(1-\theta_X(\omega_{\phi_t}))^t\omega_{\phi_t}^n\notag\\
&=&-\int_M{{\psi_{,j\bar ji}}}\psi_{,\bar i}(1-\theta_X(\omega_{\phi_t}))^t\omega_{\phi_t}^n+t\int_M{\frac{X^j_{,i}\psi_{,j}\psi_{,\bar i}}{1-\theta_X(\omega_{\phi_t})}}(1-\theta_X(\omega_{\phi_t}))^t\omega_{\phi_t}^n\notag\\
&&+t\int_M{\frac{ {\bar X^{\bar j}}X^i\psi_{,i}\psi_{,\bar j}}{(1-\theta_X(\omega_{\phi_t}))^2}}(1-\theta_X(\omega_{\phi_t}))^t\omega_{\phi_t}^n+\int_M{\frac{X^j\psi_{,\bar i}\psi_{,ij}}{1-\theta_X(\omega_{\phi_t})}}(1-\theta_X(\omega_{\phi_t}))^t\omega_{\phi_t}^n.
\end{eqnarray}
By Ricci identity and integration by parts, the first term on the right-hand side
\begin{eqnarray*}
&&-\int_M\psi_{,j\bar ji}\psi_{,\bar i}(1-\theta_X(\omega_{\phi_t}))^t\omega_{\phi_t}^n\notag\\
&=&-\int_M(\psi_{,ij\bar j}-R^p_{j~i\bar j}\psi_{,p})\psi_{,\bar i}(1-\theta_X(\omega_{\phi_t}))^t\omega_{\phi_t}^n\notag\\
&=&\int_MRic_{i\bar p}\psi_{,p}\psi_{,\bar i}(1-\theta_X(\omega_{\phi_t}))^t\omega_{\phi_t}^n+\int_M\psi_{,ij}\psi_{,\bar i\bar j}(1-\theta_X(\omega_{\phi_t}))^t\omega_{\phi_t}^n\notag\\
&&-t\int_M\psi_{,ij}\psi_{,\bar i}X^j(1-\theta_X(\omega_{\phi_t}))^{t-1}\omega_{\phi_t}^n.\notag\\
\end{eqnarray*}
Plugging the above equality into \eqref{0506}, one gets
\begin{eqnarray}\label{0507}
&&\lambda\int_M\psi_{,i}\psi_{,\bar i}(1-\theta_X(\omega_{\phi_t}))^t\omega_{\phi_t}^n\notag\\
&=&\int_M\left(Ric_{i\bar j}+{\frac{tX_{\bar j,i}}{1-\theta_X(\omega_{\phi_t})}}\right)\psi_{, j}\psi_{,\bar i}(1-\theta_X(\omega_{\phi_t}))^t\omega_{\phi_t}^n+\int_M\psi_{,ij}\psi_{,\bar i\bar j}(1-\theta_X(\omega_{\phi_t}))^t\omega_{\phi_t}^n\notag\\
&&+t\int_M{\frac{\bar X^{\bar j}X^i\psi_{,i}\psi_{,\bar j}}{(1-\theta_X(\omega_{\phi_t}))^2}}(1-\theta_X(\omega_{\phi_t}))^t\omega_{\phi_t}^n.
\end{eqnarray}
On the other hand, by \eqref{0502},
\begin{eqnarray*}
Ric_{i\bar j}(\omega_{\phi_t})=g_{i\bar j}(t)-t\left[{\frac{X_{i,\bar j}}{1-\theta_X(\omega_{\phi_t})}}+{\frac{\bar X_{,i}X_{,\bar j}}{(1-\theta_X(\omega_{\phi_t}))^2}}\right].
\end{eqnarray*}
Plugging this into (\ref{0507}), one gets
\begin{eqnarray*}
\lambda\int_M\psi_{,i}\psi_{,\bar i}(1-\theta_X(\omega_{\phi_t}))^t\omega_{\phi_t}^n=\int_Mg_{i\bar j}(0)\psi_{,\bar i}\psi_{,j}(1-\theta_X(\omega_{\phi_t}))^t\omega_{\phi_t}^n.
\end{eqnarray*}
Since $\psi\not\equiv const.$, it must holds $\lambda>0$. The lemma is proved.
\end{proof}
The openness then follows from the above lemma and implicit function theorem.

\subsection{Closedness} For the closeness, it suffices to establish the a priori estimates for \eqref{0502}.

First, we prove the $C^0$-estimate.
\begin{prop}
Let $\phi_t$ be a solution of (\ref{0502}) at $t$. Then there exists a uniform constants $C$ such that $|\phi_t|\leq C$.
\end{prop}
\begin{proof}
Consider the equation
\begin{eqnarray}\label{0509}
\det(g_{i\bar j}(t))(1-\theta_X(\omega_{\phi_t}))^t=\det(g_{i\bar j}(0))e^{h_0+J_t(\phi_t)}.
\end{eqnarray}
By integration,
\begin{eqnarray*}
\int_M(1-\theta_X(\omega_{\phi_t}))^t\omega_{\phi_t}^n=e^{J_t(\phi_t)}V,
\end{eqnarray*}
we have
\begin{eqnarray*}
J_t(\phi_t)=\log\int_M(1-\theta_X(\omega_{\phi_t}))^t\omega_{\phi_t}^n-\log V.
\end{eqnarray*}
This implies
\begin{eqnarray}\label{0509+}
t\log c_X\leq J_t(\phi_t)\leq t\log C_X.
\end{eqnarray}
(\ref{0509}) can be rewritten as
\begin{eqnarray*}\label{0510}
\det(g_{i\bar j}(t))=\det(g_{i\bar j}(0))e^{\hat f_t},
\end{eqnarray*}
where $\hat f_t:=h_0+J_t(\phi_t)-t\log(1-\theta_X(\omega_{\phi_t}))$.
Let $\hat\phi_t=\phi_t-c_t$. Then $\sup_M\hat\phi_t=-1$.  Since
$$|\hat f_t|\leq \parallel h_0\parallel_{C^0}+2\max\{|\log c_X|,|\log C_X|\},$$
by the argument of $C^0$-estimate in \cite{Tian}, we see that $|\hat\phi_t|\leq C'$ for some uniform $C'>0$. On the other hand,
\begin{eqnarray}\label{0511}
&&c_t\int_0^1\int_M(1-\theta_X(\omega_{s\phi_t}))^t\omega_{s\phi_t}^n\wedge ds\notag\\
&=&I_t(\phi_t)-\int_0^1\int_M\hat\phi_t(1-\theta_X(\omega_{s\phi_t}))^t\omega_{s\phi_t}^n\wedge ds.
\end{eqnarray}
Combining (\ref{0509+}), (\ref{0511}) and the fact that
$0<c_x\leq 1-\theta_X(\omega_{s\phi_t})\leq C_X$,
one gets a uniform constant $\hat C$ such that $|c_t|\leq\hat C$, this implies
$$|\phi_t|\leq|\hat\phi_t|+\hat C\leq C'+\hat C,$$
the proposition is proved.
\end{proof}

Next we consider the $C^2$-estimate.
\begin{prop}
Let $\phi_t$ be a solution of (\ref{0502}) at $t$. Then there exist two uniform constants $C,c$ such that
\begin{eqnarray*}
n+\bigtriangleup_{\omega_0}\phi\leq Ce^{c(\phi_t-\inf_M\phi_t)}.
\end{eqnarray*}
\end{prop}
\begin{proof}
By computation,
\begin{eqnarray*}
\bigtriangleup_{\omega_0}\log(1-\theta_X(\omega_\phi))
=-{\frac{\bigtriangleup_{\omega_0}\theta_X(\omega_\phi)}{1-\theta_X(\omega_\phi)}}-{\frac{|\partial\theta_X(\omega_\phi)|^2}{1-\theta_X(\omega_\phi)^2}}
\leq C_1(n+\bigtriangleup_{\omega_0}\phi)+C_2
\end{eqnarray*}
for some constants $C_1,C_2$ independent of $\phi$.
Then following the computations of \cite{Zhu},
\begin{eqnarray}\label{0508}
\bigtriangleup_{\omega_{\phi}}((n+\bigtriangleup_{\omega_0}\phi)e^{-c\phi})&=&e^{-c\phi}\left(\bigtriangleup_{\omega_0}(h_0-t\log(1-\theta_X(\omega_\phi)))-n^2\inf_{l\not=k}R_{i\bar il\bar l}\right)\notag\\
&&+(c+\inf_{l\neq i}R_{i\bar il\bar l})(n+\bigtriangleup_{\omega_0}\phi)e^{-c\phi}\left(\sum_i{\frac1{1+\phi_{,i\bar i}}}\right)\notag\\
&&-cn(n+\bigtriangleup_{\omega_0}\phi)e^{-c\phi}\notag\\
&\geq&-e^{-\phi}(C_3+cC_4(n+\bigtriangleup_{\omega_0}\phi))+C_5e^{-c\phi}(n+\bigtriangleup_{\omega_0}\phi)^{\frac n{n-1}}
\end{eqnarray}
for sufficiently large constant $c$ and some uniform constants $C_3\sim C_5$. The proposition then follows from (\ref{0508}) in a standard way.
\end{proof}
The higher order estimates then follow from nonlinear elliptic equation theory and we omit the details.


\end{document}